\documentclass{amsart}
\usepackage{amsfonts,amsmath,amssymb}
\usepackage{mathrsfs,mathtools}
\usepackage{graphicx,subfigure}
\usepackage{enumerate}
\usepackage{xspace}
\usepackage{hyperref}
\usepackage{bibentry}
\usepackage[norelsize]{algorithm2e}
\DeclareMathAlphabet{\mathpzc}{OT1}{pzc}{m}{it}

\newcommand{\sr}{{\tfrac12}}

\newcommand{\ve}{\mathscr{U}}
\newcommand{\R}{\mathbb{R}}

\newcommand{\Y}{\mathpzc{Y}}

\newcommand{\N}{\mathpzc{N}}
\newcommand{\M}{\mathpzc{M}}

\newcommand{\vero}{\texttt{v}}
\newcommand{\wero}{\texttt{w}}

\newcommand{\C}{\mathcal{C}}

\newcommand{\D}{\mathcal{D}}
\newcommand{\V}{\mathbb{V}}

\newcommand{\T}{\mathscr{T}}

\newcommand{\Ss}{\mathscr{S}}

\newcommand{\HL}{ \mbox{ \raisebox{7.4pt} {\tiny$\circ$} \kern-10.3pt} {H_L^1} }
\newcommand{\Wp}{ \mbox{ \raisebox{7.7pt} {\scriptsize$\circ$} \kern-10.1pt} {W^{1,p}} }
\newcommand{\Wpp}{ \mbox{ \raisebox{7.7pt} {\scriptsize$\circ$} \kern-10.1pt} {W^{1,p'}} }
\newcommand{\Sz}{ \mbox{ \raisebox{7.5pt} {\scriptsize$\circ$} \kern-10.1pt} {\Ss} }
\newcommand{\HLnew}{ \mbox{ \raisebox{7pt} {\scriptsize$\circ$} \kern-10.1pt}{H}^1_L }
\newcommand{\HLn}{{\mbox{\,\raisebox{5.1pt} {\tiny$\circ$} \kern-9.1pt}{H}^1_L  }}
\newcommand{\HLs}{{\mbox{\raisebox{8.7pt} {\scriptsize$\circ$} \kern-10.1pt}{H}^1_L  }}
\newcommand{\verot}{\emph{\texttt{v}}}
\newcommand{\Tr}{\mathbb{T}}

\DeclareMathOperator*{\tr}{tr_\Omega}

\DeclareMathOperator*{\Span}{span}

\newcommand{\Laps}{(-\Delta)^s}

\newcommand{\DIV}{\textrm{div}}
\newcommand{\diff}{\, \mbox{\rm d}}
\newcommand{\ie}{i.e.,\@\xspace}

\newcommand{\Hs}{\mathbb{H}^s(\Omega)}
\newcommand{\Ws}{\mathbb{H}^{1-s}(\Omega)}
\newcommand{\Hsd}{\mathbb{H}^{-s}(\Omega)}

\newcommand{\A}{\mathcal{A}}
\newcommand{\polN}{{\mathbb{N}}}
\newcommand{\Wmpo}{ \mbox{ \raisebox{7.4pt} {\tiny$\circ$} \kern-10.7pt} {W_p^m} }

\usepackage{color}

\newtheorem{theorem}{Theorem}[section]
\newtheorem{lemma}[theorem]{Lemma}

\newtheorem{proposition}{Proposition}[section]
\theoremstyle{definition}
\newtheorem{definition}[theorem]{Definition}

\theoremstyle{remark}
\newtheorem{remark}[theorem]{Remark}
\numberwithin{equation}{section}

\newcommand{\calT}{{\mathcal{T}}}
\newcommand{\calI}{{\mathcal{I}}}

\DeclareMathOperator*{\supp}{supp}
\newcommand{\Nin}{\,{\mbox{\,\raisebox{6.2pt} {\tiny$\circ$} \kern-11.1pt}\N }}
\newcommand{\Ninn}{{\mbox{\,\raisebox{4.5pt} {\tiny$\circ$} \kern-8.8pt}\N }}

\nobibliography*

\begin{document}

\title[Multilevel methods]{Multilevel methods for nonuniformly elliptic operators}

\author[L.~Chen]{Long Chen}
\address[L.~Chen]{Department of Mathematics, University of California at Irvine, Irvine, CA 92697, USA}
\email{chenlong@math.uci.edu}
\thanks{LC has been supported by NSF Grant DMS-1115961 and DOE prime award \# DE-SC0006903.}

\author[R.H.~Nochetto]{Ricardo H.~Nochetto}
\address[R.H.~Nochetto]{Department of Mathematics and Institute for Physical Science and Technology,
University of Maryland, College Park, MD 20742, USA}
\email{rhn@math.umd.edu}
\thanks{RHN and AJS have been supported in part by NSF grant DMS-1109325.}

\author[E.~Ot\'arola]{Enrique Ot\'arola} 
\address[E.~Ot\'arola]{Department of Mathematics, University of Maryland, College Park, MD 20742, USA}
\email{kike@math.umd.edu}
\thanks{EO has been supported in part by the Conicyt-Fulbright Fellowship Beca Igualdad de Oportunidades
and NSF grant DMS-1109325}

\author[A.J.~Salgado]{Abner J.~Salgado}
\address[A.J.~Salgado]{Department of Mathematics, University of Tennessee, Knoxville, TN 37996, USA}
\email{asalgad1@utk.edu}


\date{\today}


\begin{abstract}
We develop and analyze multilevel methods for nonuniformly elliptic operators whose ellipticity 
holds in a weighted Sobolev space with an $A_2$--Muckenhoupt weight. Using 
the so-called Xu--Zikatanov (XZ) identity, we derive a nearly uniform convergence result, under 
the assumption that the underlying mesh is quasi-uniform. We also consider the so-called $\alpha$-harmonic extension 
to localize fractional powers of elliptic operators. Motivated by the scheme proposed in 
[\bibentry{NOS}]
we present a multilevel method with line smoothers and obtain a nearly uniform convergence result on anisotropic meshes.
Numerical experiments reveal a competitive performance of our method.
\end{abstract}

\maketitle
\section{Introduction}
\label{sec:introduction}
Recently, a great deal of attention has been paid to the study of fractional and nonlocal
operators, both from the point of view of pure mathematical research as well as motivated by
several interesting applications where they constitute a fundamental part of the
modeling and simulation of complex phenomena that span vastly different length scales.

Fractional and nonlocal operators can be found in a number of applications such as boundary control problems
\cite{Duvaut}, finance \cite{Carr.Geman.ea2002,Zhang2007}, electromagnetic fluids \cite{McCN:81},
image processing \cite{GO:08}, materials science \cite{Bates}, optimization \cite{Duvaut},
porous media flow \cite{CG:93}, turbulence \cite{Bakunin}, peridynamics
\cite{DGLZ:12,DGLZ:13,Silling}, nonlocal continuum field theories \cite{Eringen} and others.
From this it is evident that the particular type of operator appearing in applications can widely vary and
that a unified analysis of their discretizations might be well beyond our reach. A more modest, but
nevertheless quite ambitious, goal is to develop an analysis  and approximation of a model operator
that is representative of a particular class. This is the purpose of our recent research program, in which
we deal with an important nonlocal operator --- fractional powers of the Dirichlet Laplace operator $\Laps$,
with $s \in (0,1)$, which for convenience we will simply call the fractional Laplacian.
The fractional Laplace operator has applications where long range or anomalous diffusion is
considered as in the flow in porous media \cite{Water:00}, or the theory of stochastic processes~\cite{Bertoin}.

In previous work \cite{NOS} we proposed a discretization technique for this operator and provided an
\emph{a priori} error analysis for it. In this paper, we shall be interested in fast multilevel methods
for the approximate solution of the discrete problems that arise from the discretization of the fractional Laplacian.
In other words, we shall be concerned with efficient solution
techniques for discretizations of the following problem. Let $\Omega$ be an open and bounded subset
of  $\R^n$ ($n\ge1$), with boundary $\partial\Omega$. Given $s\in (0,1)$ and a smooth enough function $f$,
find $u$ such that
\begin{equation}
\label{fl=f_bdddom}
  \begin{dcases}
    \Laps u = f, & \text{in } \Omega, \\
    u = 0, & \text{on } \partial\Omega.
  \end{dcases}
\end{equation}

The fractional Laplacian is a nonlocal operator (see \cite{Landkof,CS:07,CS:11}), which is one
of the main difficulties to study and solve problem \eqref{fl=f_bdddom}.
To localize it, Caffarelli and Silvestre showed in \cite{CS:07} that any power
of the fractional Laplacian in $\R^n$ can be determined as a
Dirichlet-to-Neumann operator via an extension problem on
the upper half-space $\R^{n+1}_+$.
For a bounded domain $\Omega$, this result has been adapted
in \cite{CDDS:11,ST:10}, thus obtaining an extension problem which is now posed
on the semi-infinite cylinder $\C = \Omega \times (0,\infty)$.
This extension is the following mixed boundary value problem:
\begin{equation}
\label{alpha_harm_intro}
  \begin{dcases}
    \DIV\left( y^\alpha \nabla \ve \right) = 0, & \text{in } \C, \\
    \ve = 0, & \text{on } \partial_L \C, \\
    \frac{ \partial \ve }{\partial \nu^\alpha} = d_s f, & \text{on } \Omega \times \{0\}, \\
  \end{dcases}
\end{equation}
where $\partial_L \C= \partial \Omega \times [0,\infty)$ denotes
the lateral boundary of $\C$, and
\begin{equation}
\label{def:lf}
\frac{\partial \ve}{\partial \nu^\alpha} = -\lim_{y \rightarrow 0^+} y^\alpha \ve_y,
\end{equation}
is the the so-called conormal exterior derivative of $\ve$ with $\nu$ being the unit outer normal to
$\C$ at $\Omega \times \{ 0 \}$. The parameter $\alpha$ is defined as
\begin{equation}
\label{eq:defofalpha}
  \alpha = 1-2s \in (-1,1).
\end{equation}
Finally, $d_s$ is a positive normalization constant
which depends only on $s$; see \cite{CS:07} for details. We will call $y$ the
\emph{extended variable} and the dimension $n+1$ in $\R_+^{n+1}$ the
\emph{extended dimension} of problem \eqref{alpha_harm_intro}.

The following simple strategy to find the solution of \eqref{fl=f_bdddom}
has been proposed and analyzed in \cite{NOS}: given a sufficiently smooth function $f$ we solve
\eqref{alpha_harm_intro}, thus obtaining a function $\ve = \ve(x',y)$.
Setting $u: x' \in \Omega \mapsto u(x') = \ve(x',0) \in \R$, we obtain
the solution of \eqref{fl=f_bdddom}.

For an overview of the existing numerical techniques used to solve problems involving
fractional diffusion such as the matrix transference technique
and the contour integral method, we refer to \cite{BP:13,NOS}.
In addition to \cite{NOS}, two other works that deal with the discretization of fractional powers of elliptic operators
have subsequently appeared:
the approach given by Bonito and Pasciak in \cite{BP:13} is based on the integral formulation for
self-adjoint operators discussed, for instance, in \cite[Chapter 10.4]{BS}; the work by del Teso and
V\'azquez \cite{Felix} studies the approximation of the $\alpha$-harmonic extension problem via the
finite difference method.

The main advantage of the algorithm proposed in \cite{NOS}, is that we are solving
the local problem \eqref{alpha_harm_intro} instead
of dealing with the nonlocal operator $\Laps$ of problem \eqref{fl=f_bdddom}.
However, this comes at the expense of incorporating one more dimension to the problem,
thus raising the question of how computationally efficient this approach is.
A quest for the answer is the motivation for the study of multilevel methods, since it is known
that they are the most efficient techniques for the solution of discretizations of partial differential
equations, see \cite{Brandt1977,Brandt1984,Hackbusch1985,Xu1992siamreview}. Multigrid methods for equations of the type
\eqref{alpha_harm_intro}, however, are not very well understood.

The purpose of this work is twofold and hinges on the multilevel
framework developed in \cite{BPWX:91,Xu1992siamreview} and the Xu-Zikatanov identity \cite{XuZ02}. First, we show
nearly uniform convergence of a multilevel method for a class of general nonuniformly elliptic equations
on quasi-uniform meshes. Second, we derive
an almost uniform convergence of a multilevel method for
the local problem that arises from our PDE approach to the fractional Laplacian
\eqref{alpha_harm_intro} on anisotropic meshes \cite{NOS,NOS2}.
The former result assumes that the weight $\omega$ in the differential operator belongs to
the so-called Muckenhoupt class $A_2$; see Definition~\ref{def:Muckenhoupt} for details.
A somewhat related work is \cite{Griebel2007} where the authors show a uniform norm equivalence for
a multilevel space decomposition under the assumption that the weight belongs to the smaller class $A_1$.
Their results and techniques, however, do not apply to our setting since, simply put, an
$A_1$-weight is ``almost bounded'', which is too restrictive; see Remark~\ref{rem:ApinAr} for details.
We make no regularity assumption on the weight $\omega$ and show that our estimates solely depend on
the $A_2$-constant $C_{2,\omega}$. However, our results depend on the number $J$ of levels, and thus logarithmically
on the meshsize, which seems unavoidable without further regularity assumptions. For the fractional
Laplacian, \cite{NOS} shows that a quasi-uniform mesh cannot yield quasi-optimal error estimates and,
consequently, the mesh in the extended dimension must be graded towards the bottom of the cylinder thus becoming
anisotropic. We apply line smoothers over vertical lines in the extended domain and prove that the corresponding
multigrid $\mathcal V$-cycle converges almost uniformly.

We propose an algorithm with complexity $\mathcal O(M^{n+1}\log M)$ for computing
a nearly optimal approximation of
the fractional Laplacian problem \eqref{fl=f_bdddom} in $\mathbb R^{n}$, where $M$
denotes the number of degrees of freedom in each direction. Notice that
using the intrinsic integral formulation of the fractional Laplacian  \cite{CS:11,CS:07},
a discretization would result in a dense matrix with $\mathcal O(M^{2n})$. Special techniques such as
fast multipole methods \cite{GR:87}, the $\mathcal H$-matrix methods \cite{H:99}
or wavelet methods \cite{HS:09,S:09} might be applied to reduce the complexity of storage and
manipulation of the dense matrix as well as the complexity of solvers.

The outline of this paper is as follows.
In Section~\ref{sec:Prelim}, we introduce the notation and functional framework we shall work with.
Section~\ref{sec:FEM_degenerate} contains the salient results about the finite element approximation of
nonuniformly elliptic equations including the fractional Laplacian on anisotropic meshes. Here
we also collect the relevant properties of a quasi-interpolant which are crucial to obtain the convergence
analysis of our multilevel methods.
In Section~\ref{sec:MG_degenerate}, we recall the theory of subspace corrections
\cite{Xu1992siamreview} and the Xu-Zikatanov identity \cite{XuZ02}. We present multigrid algorithms
for nonuniformly elliptic equations discretized
on quasi-uniform meshes in Section \ref{sec:MGQunif} and prove their nearly uniform convergence.
We adapt the algorithms and analysis of Section \ref{sec:MGQunif} to the fractional Laplacian discretized on anisotropic meshes
in Section~\ref{sec:MGLaps}. This requires a line smoother along the extended direction.
Finally, to illustrate the performance of our methods and the sharpness of our results, we present a series
of numerical experiments in Section~\ref{sec:Numerics}.

\section{Notation and preliminaries}
\label{sec:Prelim}

\subsection{Notation}
\label{sub:sec:notations}

Throughout this work, $\Omega$ is an open, bounded and connected subset of $\R^n$, with $n\geq1$.
The boundary of $\Omega$ is denoted by $\partial\Omega$. Unless specified otherwise, we will assume
that $\partial\Omega$ is Lipschitz.
We define the semi-infinite cylinder
\[
 \C = \Omega \times (0,\infty),
\]
and its lateral boundary
\[
 \partial_L \C = \partial \Omega \times [0,\infty).
\]
Given $\Y>0$, we define the truncated cylinder
\begin{equation}
  \label{trunccylinder}
  \C_\Y = \Omega \times (0,\Y).
\end{equation}
The lateral boundary $\partial_L\C_\Y$ is defined accordingly.

Throughout our discussion, when dealing with elements defined in $\R^{n+1}$, we shall need to
distinguish the extended dimension. A vector $x\in \R^{n+1}$, will be denoted by
\[
  x =  (x^1,\ldots,x^n, x^{n+1}) = (x', x^{n+1}) = (x',y),
\]
with $x^i \in \R$ for $i=1,\ldots,{n+1}$, $x' \in \R^n$ and $y\in\R$.
The upper half-space in $\R^{n+1}$ will be denoted by
\[
  \R^{n+1}_+ = \left\{x=(x',y): x' \in \R^n, \  y \in \R, \ y > 0 \right\}.
\]

The relation $a \lesssim b$ indicates that $a \leq Cb$, with a constant
$C$ that does not depend on $a,b,$
and the important multilevel discretization parameters $J$ and $h_{J}$
(see Section~\ref{sec:MG_degenerate} for their definitions),
but it might depend on $s$ and $\Omega$. The value of $C$ might change
at each occurrence.

If $X$ and $Y$ are topological vector spaces, we write $X \hookrightarrow Y$
to denote that $X$ is continuously embedded in $Y$. We denote by $X'$ the dual of $X$.
If $X$ is normed, we denote by $\|\cdot\|_X$ its norm.

\subsection{Weighted Sobolev spaces}
\label{sub:sec:weighted_spaces}

In the Caffarelli-Silvestre extension \eqref{alpha_harm_intro}, the parameter $\alpha=1-2s\in (-1,1)$.
Consequently, the weight $y^\alpha$ is degenerate ($\alpha>0$) or singular ($\alpha<0$),
thereby making problem \eqref{alpha_harm_intro}
nonuniformly elliptic.
The natural space for problem \eqref{alpha_harm_intro} is no longer the standard space
$H^1$ but rather the weighted
Sobolev space $H^1(y^{\alpha},\C)$, where
the weight $|y|^\alpha$ belongs to the so-called
Muckenhoupt class $A_2(\R^{n+1})$; see \cite{FKS:82,Muckenhoupt,Turesson}.
For completeness, we recall the definition of Muckenhoupt classes.

\begin{definition}[Muckenhoupt class $A_p$]
\label{def:Muckenhoupt}
Let $N \geq 1$ and $\omega \in L^1_{loc}(\R^{N})$ be such that $\omega(x) > 0$ for a.e.~$x\in \R^N$. We say that $\omega \in A_p(\R^N)$,
$1 < p < \infty$, if there exists a positive constant $C_{p,\omega}$ such that
\begin{equation}
  \label{A_pclass}
  \sup_{B} \left( \frac{1}{|B|}\int_{B} \omega \right)
            \left(\frac{1}{|B|}\int_{B} \omega^{1/(1-p)} \right)^{p-1}  = C_{p,\omega} < \infty,
\end{equation}
where the supremum is taken over all balls $B$ in $\R^N$ and $|B|$ denotes the Lebesgue measure of
$B$. In addition, we define
\[
  A_\infty(\R^N) = \bigcup_{p>1} A_p(\R^N), \quad \textrm{and} \quad
  A_1(\R^N) = \bigcap_{p>1} A_p(\R^N).
\]
\end{definition}

If $\omega$ belongs to the Muckenhoupt class $A_p(\R^N)$, we say that $\omega$ is an $A_p$-weight, and
we call the constant $C_{p,\omega}$ in \eqref{A_pclass} the $A_p$-constant of $\omega$.

\begin{remark}[characterization of the $A_1$-class]
\label{rem:ApinAr}
A useful characterization of the $A_1$-Muckenhoupt class is
given in \cite{SteinHarmonic}: $\omega \in A_1(\R^N)$ if and only if
\begin{equation}
  \sup_{B} \frac{ \| \omega^{-1} \|_{L^\infty(B)} }{|B|} \int_B \omega = C_{1,\omega} < \infty.
\label{eq:charA1}
\end{equation}
\end{remark}

Since $\alpha \in (-1,1)$, it is immediate that $|y|^{\alpha} \in A_2(\R^{n+1})$ but
$|y|^{\alpha} \notin A_1(\R^{n+1})$.

From the $A_p$-condition and H{\"o}lder's inequality follows that an $A_p$-weight satisfies the so-called
\emph{strong doubling property}. The proof of this fact is standard; see
\cite[Proposition 1.2.7]{Turesson} for more details.
\begin{proposition}[strong doubling property]
\label{pro:double}
 Let $\omega \in A_p(\R^n)$ with $1< p < \infty$ and let $E \subset \R^N$ be a measurable
subset of a ball $B \subset \R^N$. Then
\begin{equation}
\label{strong_double}
 \omega(B) \leq C_{p,\omega} \left(\frac{|B|}{|E|} \right)^p \omega(E).
\end{equation}
\end{proposition}

For a weight in the Muckenhoupt class $A_p$ we define weighted $L^p$ spaces as follows.

\begin{definition}[weighted Lebesgue spaces]
\label{Lp_weighted}
Let $\omega \in A_p$, and let $D \subset \R^N$ be an open and bounded domain. For $1 < p < \infty$,
we define the weighted Lebesgue space
$L^p(\omega, D)$ as the set of measurable functions $u$ on $D$ for which
\[
 \| u \|_{L^p(\omega, D)}= \left( \int_{D} |u|^p \omega \,dx\right)^{1/p} < \infty.
\]
\end{definition}

Based on the fact that $L^p(\omega, D) \hookrightarrow L^1_{loc}(D)$
(cf. \cite[Proposition 2.3]{NOS2}), it makes sense to talk about weak
derivatives of functions in $L^p(\omega, D)$. We define weighted Sobolev spaces as follows.

\begin{definition}[weighted Sobolev spaces]
\label{H_1_weighted}
Let $D\subset \R^N$ be an open and bounded domain,
$\omega \in A_p$ with $ 1< p < \infty$ and $m \in \mathbb{N}$. The weighted Sobolev space
$W^m_p(\omega,D)$ is the space of functions
$u \in L^p(\omega,D)$ such that for any multi-index $\kappa$ with $|\kappa|\leq m$, the weak derivatives
$D^\kappa u \in L^p(\omega, D)$. We endow $W^m_p(\omega,D)$ with the following seminorm and norm
\begin{equation*}
 |u|_{W^m_p(\omega, D)} = \left( \sum_{|\kappa| = m }
 \|D^{\kappa}u\|_{L^p(\omega, D)}^p \right)^{1/p},
\quad
 \| u \|_{W^m_p(\omega, D)} = \left( \sum_{j \leq m }
 |u|_{W^j_p(\omega, D)}^p \right)^{1/p},
\end{equation*}
respectively. We also define $\Wmpo(\omega, D)$ as the closure of $\C_0^{\infty}(D)$
in $W^m_p(\omega, D)$.
\end{definition}

Owing to the fact that $\omega \in A_p$, most of the properties of classical Sobolev spaces
have a weighted counterpart; see \cite{FKS:82,GU,Turesson}. In particular
we have the following
result (c.f.~\cite[Proposition 2.1.2, Corollary 2.1.6]{Turesson} and \cite[Theorem~1]{GU}).

\begin{proposition}[properties of weighted Sobolev spaces]
\label{PR:banach}
Let $D \subset \R^N$ be an open and bounded domain, $ 1 < p < \infty$, $\omega \in A_p(\R^N)$ and $m \in \mathbb{N}$. The spaces
\[
 W^m_p(\omega, D) \qquad \textrm{and} \qquad \Wmpo(\omega, D)
\]
are complete and $W^m_p(\omega, D) \cap \C^{\infty}(D)$ is dense in $W^m_p(\omega, D)$.
\end{proposition}

\subsection{The Caffarelli-Silvestre extension problem}
\label{sub:sec:CSextension}
Here we explore problem \eqref{alpha_harm_intro}; we refer the reader to
\cite{CS:07,CS:11,NOS} for details.
Since problem \eqref{alpha_harm_intro} is posed on the unbounded domain $\C$, it cannot be
directly approximated with finite element-like
techniques.
However, as \cite[Proposition 3.1]{NOS} shows, the solution $\ve$ decays exponentially in $y$ so that,
by truncating the cylinder $\C$ to $\C_{\Y}$
and setting a vanishing Dirichlet boundary condition on the upper boundary $y = \Y$,
we only incur in an exponentially small error in terms of $\Y$ \cite[Theorem 3.5]{NOS}.

Define
\[
  \HL(y^{\alpha},\C_{\Y}) = \left\{ v \in H^1(y^\alpha,\C_\Y): v = 0 \text{ on }
    \partial_L \C_\Y \cup \Omega \times \{ \Y\} \right\},
\]
where $\alpha = 1 -2s$. Since $y^{\alpha} \in A_2(\R^n)$, Proposition~\ref{PR:banach} shows that
$\HL(y^{\alpha},\C_\Y)$ is a Hilbert space. We also define
\begin{equation}
  \label{H}
  \Hs =
  \begin{dcases}
    H^s(\Omega),  & s \in (0,\sr), \\
    H_{00}^{1/2}(\Omega), & s = \sr, \\
    H_0^s(\Omega), & s \in (\sr,1),
  \end{dcases}
\end{equation}
which is the natural space for the solution $u$ of problem \eqref{fl=f_bdddom};
let $\Hsd$ be the dual of $\Hs$.
As \cite[Proposition 2.5]{NOS} shows, the trace operator
\[
 \HL(y^{\alpha},\C_\Y) \ni w \mapsto \tr w \in \Hs,
\]
is well defined. Problem \eqref{alpha_harm_intro} then reads: find $v \in \HL(y^{\alpha},\C_\Y)$ such that
\begin{equation}
\label{alpha_harmonic_extension_weak_T}
  \int_{\C_\Y} y^\alpha \nabla v \cdot \nabla \phi = d_s \langle f, \tr \phi \rangle_{\Hsd \times \Hs},
    \quad \forall \phi \in \HL(y^{\alpha},\C_\Y),
\end{equation}
where, $\langle \cdot, \cdot \rangle_{\Hsd \times \Hs}$ denotes the duality pairing between $\Hs$ and
$\Hsd$. Finally, we recall the exponential convergence result \cite[Theorem 3.5]{NOS}:
\begin{equation*}
  \| \nabla(\ve - v) \|_{L^2(\C, y^{\alpha})} \lesssim e^{-\sqrt{\lambda_1} \Y/4} \| f\|_{\Hsd},
\end{equation*}
where $\lambda_1$ denotes the first eigenvalue of the Dirichlet Laplace operator and
$\Y$ is the truncation parameter.


\section{Finite element discretization of nonuniformly elliptic equations}
\label{sec:FEM_degenerate}
Let $D$ be an
open and bounded subset of $\R^N$ ($N\ge1$)
with boundary $\partial D$ and let $f \in L^2(\omega^{-1},\D)$.
In this section, we focus on the study of a finite element method for the following
nonuniformly elliptic boundary value problem: find $u \in H_0^1(\omega,\D)$ that solves
\begin{equation}
\label{weighted_second}
  \begin{dcases}
    -\DIV (\A(x) \nabla u) = f, & \text{in } D, \\
    u = 0, & \text{on } \partial D,
  \end{dcases}
\end{equation}
where
$\A : D \rightarrow \R^{N\times N}$ is symmetric and satisfies the following nonuniform ellipticity condition
\begin{equation}
\label{weight_singular}
\omega(x) |\xi|^2 \lesssim \xi^\intercal \A(x) \xi \lesssim \omega(x) |\xi|^2,
\qquad \forall \xi \in \R^N, \quad \text{a.e.}\ x \in D.
\end{equation}
The function $\omega$ belongs to the Muckenhoupt class $A_2$, which is defined by \eqref{def:Muckenhoupt}.
Examples of nonuniformly elliptic equations are the harmonic extension problem related with the
fractional Laplace operator~\cite{CS:11,CS:07}, elliptic PDEs in an axisymmetric three dimensional domain
with axisymmetric data~\cite{BBD:06,GP:06}, and equations modeling the motion of particles in a central potential
field in quantum mechanics~\cite{ABH:06}.

Nonuniformly elliptic elliptic equations of the type \eqref{weighted_second}-\eqref{weight_singular}
have been studied in
\cite{FKS:82}. Given $f \in L^2(\omega^{-1},D)$, there exists a unique
solution $u \in H_0^1(\omega,D)$~\cite[Theorem 2.2]{FKS:82}.
Notice that by taking the weight $\omega$ to be $y^{\alpha}$ we see that the
underlying differential operator in \eqref{alpha_harmonic_extension_weak_T} is a
particular instance of $-\DIV (\A(x) \nabla u)$ in \eqref{weighted_second}.

We define the bilinear form
\begin{equation}
  a(u,v) = \int_D  \A \nabla u\cdot \nabla v \diff x,
\label{eq:defofaform}
\end{equation}
which is clearly continuous and coercive in
$H_0^1(\omega,D)$. Then, a weak formulation of problem
\eqref{weighted_second} reads: find $u\in H^1_0(\omega,D)$ such that
\begin{equation}
\label{degenerate_weak}
  a(u,v) = \int_D f v \diff x, \quad \forall v \in H^1_0(\omega,D).
\end{equation}

\subsection{Finite element approximation on quasi-uniform meshes}
\label{sub:FE}
Let us describe the construction of the underlying finite element spaces.
To avoid technical difficulties, we assume $D$ to be a polyhedral domain. Let $\T = \{ T \} $
be a mesh of $D$
into elements $T$ (simplices or cubes) such that
\[
  \bar D = \bigcup_{T \in \T} T, \qquad
  |D| = \sum_{T \in \T} |T|.
\]
The partition $\T$ is assumed to be conforming or compatible, \ie the
intersection of any two cells $T$ and $T'$ in $\T$ is either empty or a common lower
dimensional element. We denote by $\Tr$ the
collection of all conforming meshes. We say that
$\Tr$ is \textit{shape regular} if there exists a constant $\sigma > 1$ such that,
for all $\T \in \Tr,$
\begin{equation}
\label{shape_reg}
 \max \left\{ \sigma_T : T \in \T \right\} \leq \sigma,
\end{equation}
where $\sigma_T:=h_T/\rho_T$ is the shape coefficient of $T$. For simplicial elements,
$h_T = \textrm{diam}(T)$ and $\rho_T$ is the diameter of the largest sphere inscribed in $T$
\cite{BrennerScott,Ciarletbook}.
For the definition of $h_T$ and $\rho_T$ in the case of $n$-rectangles, we refer to~\cite{Ciarletbook}.

We assume that the collection of meshes $\Tr$ is conforming and satisfies the regularity assumption \eqref{shape_reg},
which says that the element shape
does not degenerate with refinement. A refinement method
generating meshes satisfying the shape regular
condition \eqref{shape_reg} will be
called \emph{isotropic refinement}. A particular instance of an isotropic refinement
is the so called quasi-uniform refinement. We recall that $\Tr$ is quasi-uniform if it is shape regular and
for all $\T \in \Tr$ we have
\[
  \max \left\{ h_T: T \in \T  \right\} \lesssim \min \left\{ h_T: T \in \T \right\},
\]
where the hidden constant is independent of $\T$. In this case, all the elements on the same refinement level are of comparable size. We define $h_{\T}= \max_{T \in \T}h_T$.

Given a mesh $\T \in \Tr$, we define
the finite element space of continuous piecewise polynomials of degree one
\begin{equation}
\label{fe_space}
  \V(\T) = \left\{
            W \in \C^0( \bar{D}): \ W|_T \in \mathcal{P}(T) \ \forall T \in \T, \
            W|_{\partial \Omega} = 0
          \right\},
\end{equation}
where for a simplicial element $T$, $\mathcal{P}(T)$ corresponds to the space
of polynomials of total degree at most one, \ie $\mathbb{P}_1(T)$,
and for $n$-rectangles, $\mathcal{P}(T)$ stands for the space of polynomials
of degree at most one in each variable, \ie $\mathbb{Q}_1(T)$.

The finite element approximation of $u$, solution of problem \eqref{weighted_second},
is defined as the unique discrete function $U_{\T} \in \V(\T)$ satisfying
\begin{equation}\label{weighted_discrete_2}
  a(U_\T,W) = \int_D f W, \quad \forall W \in \V(\T).
\end{equation}

\subsection{Quasi-interpolation operator}
\label{sub:interpolation_operator}
Let us recall the main properties of the quasi-interpolation operator $\Pi_{\T}$ introduced and analyzed
in~\cite{NOS2}. This operator is based on local averages over stars, and
then it is well defined for functions in $L^p(\omega,D)$.
We summarize its construction and its
approximation properties as follows; see~\cite{NOS2} for details.

Given a mesh $\T \in \Tr$ and $T\in \T$, we denote by $\N(T)$
the set of nodes of $T$. We set $\N(\T):= \cup_{T \in \T} \N(T)$ and
$\Nin(\T):= \N(\T) \cap D.$ Then, any discrete function $W \in \V(\T)$ is characterized
by its nodal values on the set $\Nin(\T)$. Moreover, the functions $\phi_{\vero} \in \V(\T)$, $\vero \in \Nin(\T)$,
such that $\phi_{\vero}(\wero) = \delta_{\vero \wero}$ for all $\wero \in \N(\T)$ are the canonical basis of $\V(\T)$, and
\[
 W = \sum_{\vero \in \Ninn(\T)} W(\vero) \phi_{\vero}.
\]

Given a vertex $\vero \in \N(\T)$, we define the star or patch around
$\vero$ as
$
  S_{\vero} = \cup_{T \ni \vero} T,
$
and for $T \in \T$ we define its patch as
$
  S_T = \cup_{\vero \in T} S_\vero.
$
For each vertex $\vero \in \N(\T)$, we define $h_{\vero}=\min\{h_{T}: \vero \in T \}$.

Let $\psi \in \C^{\infty}(\R^N)$ be such that $\int \psi = 1$ and $\textrm{supp } \psi \subset B$,
where $B$ denotes the ball in $\R^N$ of radius $r$ centered at zero with
$r \leq 1/\sigma$, with $\sigma$ defined by \eqref{shape_reg}.
For $\vero \in \Nin(\T)$, we define the rescaled smooth function
\[
  \psi_{\vero}(x) =  \frac{1}{ h_{\vero}^N } \psi \left(\frac{\vero-x}{h_{\vero}}\right).
\]

Given a smooth function $v$, we denote by $P^1v(x,z) $ the Taylor polynomial of degree one of the function $v$
in the variable $z$ about the point $x$, \ie
\[
  P^1v(x,z) = v(x) + \nabla v(x) \cdot (z - x).
\]
Then, given $\vero \in \Nin(\T)$ and a function  $v \in W^1_p(\omega,D)$,
we define the corresponding averaged Taylor polynomial of first degree of $v$ about the vertex $\vero$ as
\begin{equation}
\label{p1average}
  Q^1_{\vero}v(z) = \int P^1v(x,z) \psi_{\vero}(x) \diff x.
\end{equation}
Since $\textrm{supp } \psi_{\vero} \subset S_{\vero}$,
the integral appearing in \eqref{p1average} can be written over $S_{\vero}$. Moreover, integration by parts
shows that $Q^1_{\vero}$ is well defined for functions in $L^1(D)$; see~\cite[Proposition 4.1.12]{BS}.
Consequently,~\cite[Proposition 2.3]{NOS} implies that $Q^1_{\vero}$ is also well defined for functions
in $L^p(\omega,D)$ with $\omega \in A_p(\R^N)$.

Given $\omega \in A_p(\R^N)$ and $v \in L^p(\omega,D)$, we define the quasi-interpolant $\Pi_{\T}v$
as the unique function  $\Pi_\T v \in \V(\T)$ that satisfies $\Pi_{\T}v(\vero) = Q^1_{\vero}(\vero)$
if $\vero \in \N(\T)$, and $\Pi_{\T}v(\vero) = 0$ if $\vero \in \N(\T) \cap \partial \Omega$, \ie
\[
 \Pi_{\T} v(\vero) = \sum_{\vero \in \Ninn(\T)} Q^1_{\vero}(\vero) \phi_{\vero}.
\]

For this operator,~\cite[Section 5]{NOS2} proves stability and interpolation error estimates in the weighted $L^p$-norm
and $W^1_p$-seminorm. We recall these results for completeness.

\begin{proposition}[weighted stability and local error estimate I]
\label{TH:v - PivL2}
Let $T \in \T$, $\omega \in A_p(\R^N)$ and $v \in L^p(\omega,S_T)$. Then, we have the following local stability bound
\begin{equation}
\label{PiLpbounded}
  \|  \Pi_{\T} v \|_{L^p(\omega,T)} \lesssim \| v \|_{L^p(\omega,S_{T})}.
\end{equation}
If, in addition, $v \in W^1_p(\omega,S_{T})$, then we have the local interpolation error estimate
\begin{equation}
\label{v - PivLp}
  \| v - \Pi_{\T} v \|_{L^p(\omega,T)} \lesssim h_{\verot}
      \| \nabla v\|_{L^p(\omega,S_{T})}.
\end{equation}
The hidden constants in both inequalities depend only on $C_{p,\omega}$, $\psi$ and $\sigma$.
\end{proposition}

\begin{proposition}[weighted stability and local error estimate II]
\label{TH:v - PivH1}
Let $T \in \T$, $\omega \in A_p(\R^N)$ and $v \in W^1_p(\omega,S_T)$. Then, we have the following local stability bound
\begin{equation}
\label{dPIcontinuous}
  \|\nabla \Pi_{\T} v \|_{L^p(\omega,T)} \lesssim \| \nabla v\|_{L^p(\omega,S_T)}.
\end{equation}
If, in addition, $v \in W^2_p(\omega,S_T)$, then
\begin{equation}
\label{v - PivW1p}
  \|\nabla( v - \Pi_{\T} v) \|_{L^p(\omega,T)} \lesssim h_{\verot}  \| D^2 v \|_{L^p(\omega,S_T)}.
\end{equation}
The hidden constants in both inequalities depend only on $C_{p,\omega}$, $\psi$ and $\sigma$.
\end{proposition}

\subsection{Finite element approximation on anisotropic meshes}
\label{sub:anisoFE}
Let us now focus our attention on the finite element discretization of problem \eqref{alpha_harm_intro}.
To do so, we must first study the regularity of its solution $\ve$.
An error estimate for $v$, solution of
  \eqref{alpha_harmonic_extension_weak_T}, depends on the regularity
  of $\ve$ as well~\cite[\S 4.1]{NOS}.
The second order regularity of $\ve$
is much worse in the extended direction as the following estimates from~\cite[Theorem 2.7]{NOS} reveal
\begin{align}
    \label{reginx}
  \| \Delta_{x'} \ve\|_{L^2(y^{\alpha},\C)} +
  \| \partial_y \nabla_{x'} \ve \|_{L^2(y^{\alpha},\C)}
  & \lesssim \| f \|_{\Ws}, \\
\label{reginy}
  \| \ve_{yy} \|_{L^2(y^{\beta},\C)} &\lesssim \| f \|_{L^2(\Omega)},
\end{align}
where $\beta > 2\alpha + 1$.
This suggests that \emph{graded} meshes in the extended variable $y$ play a fundamental role.

Estimates \eqref{reginx}-\eqref{reginy} motivate the construction of a mesh over $\C_{\Y}$
with cells of the form $T = K \times I$, where $K \subset \R^n$ is an element
that is isoparametrically equivalent either to the unit cube $[0,1]^n$ or the unit simplex in $\R^n$ and
$I \subset \R$ is an interval. To be precise,
let $\T_\Omega = \{T\}$ be a conforming and shape regular mesh of $\Omega$.
In order to obtain a global regularity assumption for $\T_{\Y}$, we assume that
there is a constant $\sigma_{\Y}$ such that
if $T_1=K_1\times I_1$ and $T_2=K_2\times I_2 \in \T_\Y$ have nonempty intersection, then
\begin{equation}
\label{shape_reg_weak}
     \frac{h_{I_1}}{h_{I_2}} \leq \sigma_{\Y},
\end{equation}
where $h_I = |I|$. Exploiting the Cartesian structure of the mesh it is possible to handle
anisotropy in the extended variable and obtain estimates of the form
\begin{align*}
  \| v - \Pi_{\T_\Y} v \|_{L^2(y^\alpha,T)} & \lesssim h_{\vero'}  \| \nabla_{x'} v\|_{L^2(y^\alpha,S_T)} +
  h_{\vero''}\| \partial_y v\|_{L^2(y^\alpha,S_T)}, \\
  \| \partial_{x_j}(v - \Pi_{\T_\Y} v) \|_{L^2(y^\alpha,T)} &\lesssim
 h_{\vero'}  \| \nabla_{x'} \partial_{x_j} v\|_{L^2(y^\alpha,S_T)} +
  h_{\vero''}\| \partial_y \partial_{x_j} v\|_{L^2(y^\alpha,S_T)},
\end{align*}
with $j=1,\ldots,n+1$, where $h_{\vero'}= \min\{h_K: \vero' \in  K\}$,
$h_{\vero''} = \min\{h_I: \vero'' \in I\}$ and $v$ is the solution
of problem \eqref{alpha_harmonic_extension_weak_T}; see~\cite[\S 4.2.3 and 4.2.4]{NOS} for details.
However, since $\ve_{yy} \approx y^{-\alpha -1 }$ as $y \approx 0$,
we realize that $\ve \notin H^2(y^{\alpha},\C)$
and the second estimate is not meaningful for $j=n+1$.
In view of the regularity estimate \eqref{reginy}
it is necessary to measure the regularity of $\ve_{yy}$ with a
different weight and thus compensate with a graded mesh in the extended
dimension. This makes anisotropic estimates essential.

In order to simplify the analysis and
implementation of multilevel techniques, we consider a sequence
of nested discretizations. We construct such meshes as follows.
First, we introduce a sequence of nested uniform partitions of the
unit interval $\{ \calT_k \}$,  with mesh points
${\widehat y}_{l,k}$, for $l=0,\ldots,M_k$ and $k=0,\dots,J$. Then, we obtain a family of meshes of the interval
$[0,\Y]$ given by the mesh points
\begin{equation}
\label{graded_mesh}
  y_{l,k} = \Y {\widehat y}_{l,k}^\gamma, \quad l=0,\dots,M_k,
\end{equation}
where $\gamma > 3/(1-\alpha)$. Then, for $k = 0,\dots,J$, we consider
a quasi-uniform triangulation $\T_{\Omega,k}$ of the domain $\Omega$
and construct the mesh $\T_{\Y,k}$ as the tensor product of $\T_{\Omega,k}$
and the partition given in \eqref{graded_mesh};
hence $\#\T_{\Y,k} = M_k \, \# \T_{\Omega,k}$. Assuming that $\# \T_{\Omega,k} \approx M^n_k$
we have $\#\T_{\Y,k} \approx M_k^{n+1}$. Finally, since $\T_{\Omega,k}$ is shape regular and quasi-uniform,
$h_{\T_{\Omega,k}} \approx (\# \T_{\Omega,k})^{-1/n}$. All these considerations allow us to obtain the
following result \cite[Theorem 5.4 and Remark 5.5]{NOS}.

\begin{theorem}[error estimate]
\label{TH:fl_error_estimates}
Denote
by $V_{\T_{\Y,k}} \in \V(\T_{\Y,k})$ the Galerkin approximation
of problem \eqref{alpha_harmonic_extension_weak_T} with first degree
tensor product elements. Then,
\begin{equation*}
  \| \nabla(\ve - V_{\T_{\Y,k}}) \|_{L^2(y^\alpha,\C)} \lesssim
|\log(\# \T_{\Y,k})|^s(\# \T_{\Y,k})^{-1/(n+1)} \|f \|_{\mathbb{H}^{1-s}(\Omega)},
\end{equation*}
where $\Y \approx \log(\# \T_{\Y,k})$.
\end{theorem}

We notice that the anisotropic meshes of the cylinder $\C_\Y$ considered above
are semi-structured by construction. They are generated as the tensor
product of an unstructured grid $\T_{\Omega}$ together with the structured mesh
$\mathcal{T}_k$. Figure~\ref{fig:ctm} shows an example of this type of meshes in three dimensions.

\begin{figure}[ht]
\includegraphics[scale=0.44]{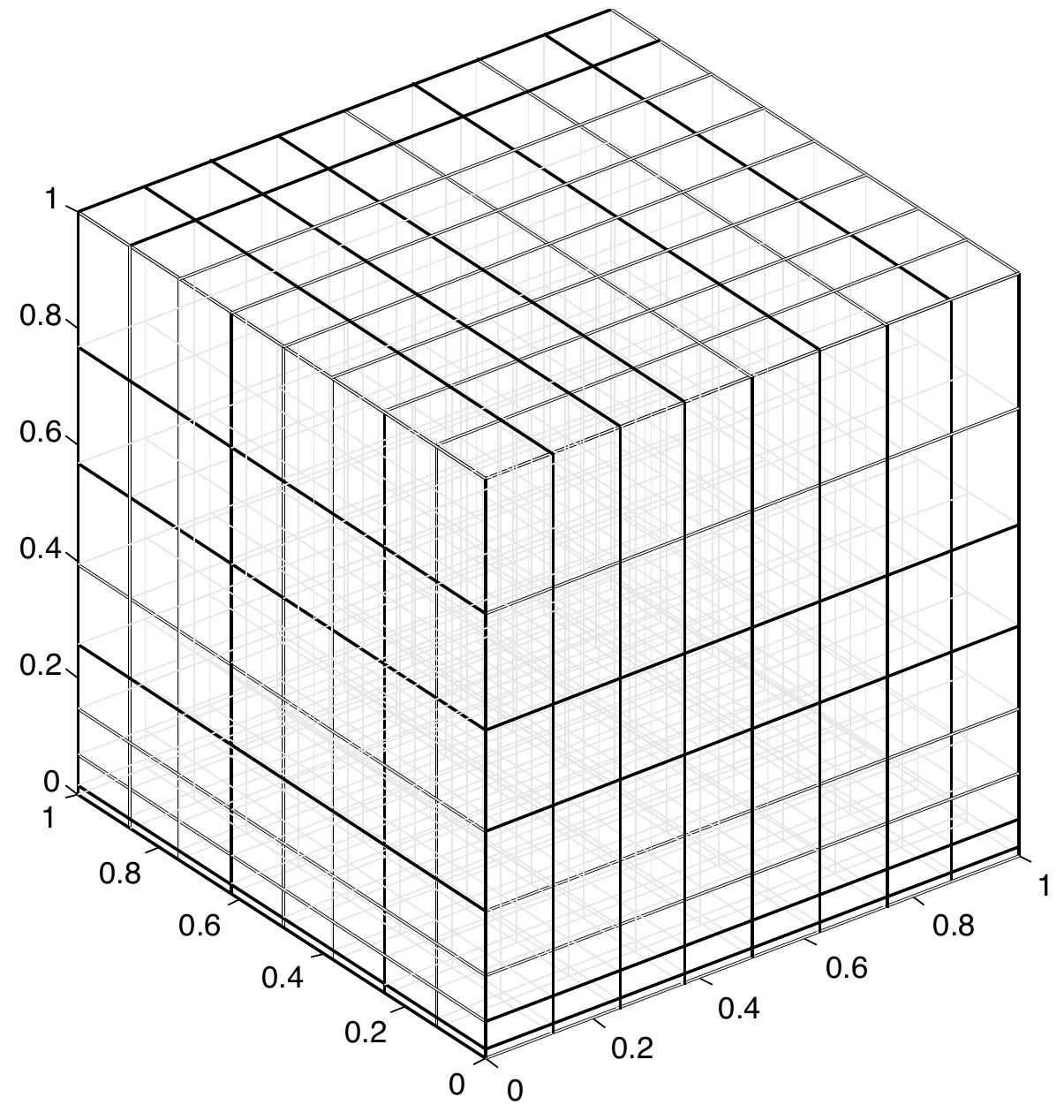}
\caption{
A three dimensional graded mesh of the cylinder
$(0,1)^2 \times (0,\Y)$ with $392$ degrees of freedom. The mesh is constructed
as a tensor product of a quasi-uniform mesh of $(0,1)^2$ with cardinality $49$
and the image of the quasi-uniform partition of the interval $(0,1)$
with cardinality $8$ under the mapping \eqref{graded_mesh}. }
\label{fig:ctm}
\end{figure}

Notice that the approximation
estimates \eqref{PiLpbounded}-\eqref{v - PivW1p} are local and
thus valid under the weak shape regularity condition \eqref{shape_reg_weak}.
Owing to the tensor product structure of the mesh, we have the following anisotropic error estimate.

\begin{lemma}[weighted $L^2$ anisotropic error estimate]
\label{le:l2aniso}
Let $v \in \HL(y^\alpha, \C_\Y)$ be  the solution of problem \eqref{alpha_harmonic_extension_weak_T}. Then, the quasi-interpolation
operator $\Pi_{\T_\Y}$ satisfies the following error estimate
\[
\| v- \Pi_{\T_\Y} v\|_{L^2(y^{\alpha}, \C_\Y)} \lesssim \# \T_{\Y}^{-1/(n+1)}
\left(
  \| \nabla_{x'}v \|_{L^2(y^{\alpha},\C_{\Y})} + \|  \partial_{y} v \|_{L^2(y^{\alpha},\C_{\Y})}
\right).
\]
\end{lemma}
\begin{proof}
This is a direct consequence of the results from~\cite[\S 5]{NOS2} together with the Cartesian
structure of the mesh $\T_{\Y}$.
\end{proof}

A simple application of the mean value theorem yields
\begin{equation}
  y_{l+1,k} - y_{l,k} = \frac\Y{M_k^\gamma} \big( (l+1)^\gamma - l^\gamma \big)
  \leq \gamma \frac\Y{M_k} \left( \frac{l+1}{M_k} \right)^{\gamma-1} \leq \gamma \frac\Y{M_k},
\label{eq:mesh_rels}
\end{equation}
for every $l=0,\ldots, M_k-1$, where $\gamma > 3/(1-\alpha) = 3/(2s)$ according to \eqref{graded_mesh}.
In other words, since the meshsize of the quasi-uniform mesh $\T_{\Omega,k}$ is $\mathcal{O}(M_k^{-1})$, the size of the partitions
in the extended variable $y$ can be uniformly controlled by $h_{\T_{\Omega,k}}$ for $k=0,\dots,J$. However,
$\gamma$ blows up as $s \downarrow 0$.

\section{Multilevel space decomposition and multigrid methods}
\label{sec:MG_degenerate}
In this section, we present a $\mathcal V$-cycle multigrid algorithm based
on the method of subspace corrections \cite{BPWX:91,Xu1992siamreview},
and we present the key identity of Xu and Zikatanov \cite{XuZ02}
in order to analyze the convergence of the proposed multigrid algorithm.

\subsection{Multilevel decomposition}
\label{sub:multilevel}
We follow \cite{BP:87,BPWXreg:91} to present a multilevel decomposition of the space $\V(\T)$.
Assume that we have an initial conforming mesh $\T_0$ made
of simplices or cubes, and a nested
sequence of discretizations $\{ \T_k \}_{k=0}^{J}$ where, for $k>0$,
$\T_k$ is obtained by uniform refinement of $\T_{k-1}$.
We then obtain a nested sequence,
in the sense of trees, of quasi-uniform meshes
\[
  \T_0 \leq \T_1 \leq \dots \leq \T_J = \T.
\]
Denoting by $h_k := h_{\T_{k}}$ the meshsize of the mesh $\T_k$, we have that
$h_k \eqsim \rho^{k}$ for some $\rho \in (0,1)$, and then $J \eqsim |\log h_J|$.
Let $\V_k:= \V(\T_k)$ denote the corresponding
finite element space over $\T_k$ defined by \eqref{fe_space}. We thus
get a sequence of nested
spaces
\[
  \V_0 \subset \V_1 \subset \dots \subset \V_J = \V,
\]
and a macro space decomposition
\[
  \V = \sum_{k=0}^{J} \V_k.
\]
Note the redundant overlapping of the multilevel decomposition above; in particular,
the sum is not direct. We now introduce a space micro-decomposition. We start by defining
$\N_{\,k}:= \N(\T_k) = \dim \V_k$, \ie the number of interior vertices of the mesh $\T_k$.
In order to deal with point and line Gauss-Seidel smoothers, we introduce the following sets of indices:
For $j=1,\ldots,\M_k$ we denote by
$\calI_{k,j}$ a subset of the index set $\{1, 2, \ldots , \N_k\}$, and assume $\calI_{k,j}$ satisfies
\[
 \bigcup_{j=1}^{\M_k} \calI_{k,j} = \{1, 2, \ldots , \N_k\}.
\]
The sets $\calI_{k,j}$ may overlap, \ie given $0<j_1, j_2 \leq \M_k$
such that $j_1 \neq j_2$, we may have $\calI_{k,j_1} \cap \calI_{k,j_2} \neq \emptyset$. This overlap, however,
is finite and independent of $J$ and $\N_{\,k}$.

Upon denoting the standard nodal basis of $\V_k$ by $\phi_{k,i}$, $i=1,\dots,\N_k$,
we define $\V_{k,j} = \Span \{ \phi_{k,i}: i\in \calI_{k,j}\}$ and we have the space decomposition
\begin{equation}
\label{V=decomp}
  \V = \sum_{k=0}^{J} \sum_{j=1}^{\M_k} \V_{k,j}.
\end{equation}

\subsection{Multigrid algorithm}
\label{sub:multigrid}
We now describe the multigrid algorithm for the non-uniformly elliptic problem
\eqref{weighted_second}. We start by introducing several auxiliary operators.
For $k=0,\ldots,J$, we define
the operator $A_k: \V_k \rightarrow \V_k$ by
\[
  (A_k v_k,w_k)_{L^2(\omega,D)} = a(v_k,w_k), \quad \forall v_k,w_k \in \V_k,
\]
where the bilinear form $a$ is defined in \eqref{eq:defofaform}.
Notice that this operator is symmetric and positive definite with respect to the weighted $L^2$-inner
product.
The projection operator $P_k: \V_J \rightarrow \V_k$ in the $a$-inner product is
defined by
\[
  a(P_k v,w_k) = a(v,w_k), \quad \forall w_k \in \V_k,
\]
and the weighted $L^2$-projection $Q_k: \V_J \rightarrow \V_k$ is defined by
\[
  (Q_k v,w_k)_{L^2(\omega,D)} = (v,w_k)_{L^2(\omega,D)}, \quad w_k \in \V_k.
\]

We define, analogously, the operators
$A_{k,j}:\V_{k,j} \rightarrow \V_{k,j}$,
$P_{k,j}:\V_k \rightarrow \V_{k,j}$ and
$Q_{k,j}:\V_k \rightarrow \V_{k,j}$.
The operator $A_{k,j}$ can be regarded as the restriction of $A_k$ to the subspace $\V_{k,j}$, and
its matrix representation,
which is the sub-matrix of $A_k$ obtained by deleting the indices $i \notin \calI_{k,j}$,
is symmetric and positive definite.
On the other hand, the operators $P_{k,j}$ and $Q_{k,j}$ denote the
projections with respect to the $a$- and the weighted $L^2$-inner
products into $\V_{k,j}$, respectively.
We also remark that the matrix representation of the operator $Q_{k,j}$
is the so-called restriction operator, and
the prolongation operator $Q_{k,j}^T$ corresponds to the natural embedding $\V_{k,j} \hookrightarrow \V_k$.
The following property, which is of fundamental importance, will be used frequently in the
paper
\begin{equation}
\label{eq:commutator}
  A_{k,j}P_{k,j} = Q_{k,j}A_k.
\end{equation}

With this notation we define a symmetric
$\mathcal{V}$-cycle multigrid method as in Algorithm~\ref{alg:Vcycle}.
When $m=1$, it is equivalent to the application of successive subspace
corrections (SSC) to the decomposition \eqref{V=decomp} with
exact sub-solvers $A_{k,j}^{-1}$ so that
the $\mathcal{V}$-cycle multigrid method has a smoother at each level of block Gauss-Seidel
type \cite{BPWXreg:91,Xu1992siamreview}. In particular, if we consider
a nodal decomposition $\calI_{k,j} = \{j\}$ we obtain a point-wise Gauss-Seidel smoother.
On the other hand, if the indices in $\calI_{k,j}$ are such that the corresponding vertices lie on a straight line,
we obtain the so-called line Gauss-Seidel smoother, which will be
essential to efficiently solve problem \eqref{alpha_harm_intro}
with anisotropic elements.

\begin{algorithm}
\label{alg:Vcycle}
\SetKwData{MG}{MG}
\SetKwInOut{Input}{input}\SetKwInOut{Output}{output}

$e = \MG(r,k,m)$

\Input{$r \in \V_k$ --- residual; \\
       $k \in \{0,\ldots J \}$ --- level; \\
       $m \in \polN$ --- number of smoothing steps.}
\Output{$e \in \V_k$ --- an approximate solution of $A_k e = r$.}

\If{$k=0$}{
  $e=A_{0}^{-1}r$\;
}

\tcp{pre-smoothing: $m$ steps}

$u^0 = 0$\;

\For{$l\leftarrow 1$ \KwTo $m$}{
  $v\leftarrow u^{l-1}$\;
  \For{$j\leftarrow 1$ \KwTo $\M_k$}{
    $v\leftarrow v+A_{k,j}^{-1}Q_{k,j}(r-A_{k}v)$\;
  }
  $u^{l}\leftarrow v$\;
}

\tcp{coarse grid correction}

$u^{m+1} = u^m + \MG\left(Q_{k-1}(r-A_{k}u^{m}), k-1, m\right)$\;

\tcp{post-smoothing: $m$ steps}

\For{$l\leftarrow m+2$ \KwTo $2m+1$}{
  $v\leftarrow u^{l-1}$\;
  \For{$j \leftarrow  \M_{k}$ \KwTo $1$}{
    $v\leftarrow v+A_{k,j}^{-1}Q_{k,j}(r-A_{k}v)$\;
  }
  $u^{l}\leftarrow v$\;
}

\tcp{output}
$e=u^{2m+1}$\;

\caption{Symmetric $\mathcal V$-cycle multigrid method}
\end{algorithm}

\subsection{Analysis of the multigrid method}
\label{sub:mganalysisdegenerate}
In order to prove the nearly uniform convergence of the symmetric $\mathcal{V}$-cycle multigrid method
without any assumptions,
we rely on the following fundamental identity developed by Xu and Zikatanov \cite{XuZ02}; see also
\cite{ChenXZ,CXZ:08} for alternative proofs.

\begin{theorem}[XZ Identity]
\label{XZ}
Let $\V$ be a Hilbert space with inner product $(\cdot,\cdot)_A$ and norm $\|\cdot\|_A$.
Let $\V_j \subset \V$ be a closed subspace of $\V$ for $j=0,\ldots,J$, satisfying
\[
  \V  = \sum_{j=0}^{J} \V_j.
\]
Denote by $P_j: \V \rightarrow \V_j$ the orthogonal projection in the inner product
$(\cdot,\cdot)_A$ onto $\V_j$. Then, the following identity holds
\[
  \left \|  \prod_{j=0}^J (I-P_j) \right\|_{A}^2 = 1 -\frac{1}{1+c_0},
\]
where
\begin{equation}
\label{c_0}
c_0 = \sup_{\|\nu\|_{A}=1} \inf_{\sum_{i=0}^J \nu_i = \nu}
    \sum_{i=0}^J \left\| P_i \sum_{j=i+1}^J \nu_j \right\|_{A}^2.
\end{equation}
\end{theorem}

The XZ identity given by Theorem~\ref{XZ}, the properties of the interpolation operator $\Pi_{\T}$ defined in
\S\ref{sub:interpolation_operator}, the stability of the nodal decomposition stated in Lemma~\ref{lem:microStab} below,
and the weighted inverse inequality proved in Lemma~\ref{le:inv_est}
below, will allow us to obtain the nearly uniform convergence of the symmetric $\mathcal V$-cycle
multigrid method described in Algorithm \ref{alg:Vcycle}, without resorting to any regularity
assumptions on the solution.
To see how this is possible we recall the basic ingredients in the analysis of multilevel methods;
see \cite{BP:87, BPWXreg:91, MR1804746,Xu1992siamreview} for details.
We introduce, for $k=1,2,\ldots,J$, the operator
\begin{align*}
  K_k &= (I-A_{k,\M_k}^{-1}Q_{k,\M_k}A_{k}) \cdots (I-A_{k,1}^{-1}Q_{k,1}A_{k}) \\
      &= (I-P_{k,\M_k})\cdots (I-P_{k,1})
      = \prod_{j=1}^{\M_{k}}(I-P_{k,j}),
\end{align*}
where we used \eqref{eq:commutator} to obtain the second equality.
With this notation, Algorithm~\ref{alg:Vcycle} can then be recast as a two-layer iterative scheme
for the solution of $A_J u = f$ of the form
\[
  u^{\ell+1} = u^\ell + B_J ( f - A_J u^\ell ),
\]
where the iterator $B_J$ satisfies
\[
  I-B_J A_J = \left(K_J^m\right)^* \cdots \left(K_1^m\right)^*(I-P_0)
              K_1^m \cdots K_J^m,
\]
with $M^{*}$ denoting the adjoint operator of $M$ with respect to the $a$--inner product. Notice
that $I - B_J A_J$ is the so-called error transfer operator \ie
\[
  u - u^{\ell+1} = \left( I - B_J A_J \right)\left( u - u^\ell \right).
\]
Consequently, to show convergence of our scheme we must show that this operator is a contraction
with a contraction factor, ideally, independent of $J$. Owing to the fact that
\[
  \|K_k^m\|_{A}\leq\|K_k\|^m_{A}\leq\|K_k\|_{A},
\]
it suffices to consider the case $m=1$, where, given an operator $S$, we denote by $\| S \|_A$ the operator norm induced by
the bilinear form $a$.
Therefore
\begin{equation}
\label{eq:erroroperator}
  \| I - B_J A_J \|_A
                      \leq \left\| \prod_{k=0}^J \prod_{j=1}^{\M_{k}}(I-P_{k,j}) \right\|_A^2,
\end{equation}
because $P_0$ is an exact solve and thus a projection, whence $(I-P_0)^2 = I-P_0$.
Notice that the right hand side of \eqref{eq:erroroperator} is precisely the quantity that the XZ identity
provides a value for. In conclusion, based on Theorem~\ref{XZ}, to prove
the convergence of the symmetric $\mathcal{V}$-cycle multigrid method described in Algorithm \ref{alg:Vcycle},
we must obtain an estimate for the constant $c_0$ given by \eqref{c_0}, which will be the content
of the next two sections.

\section{Analysis of Multigrid Methods on Quasi-uniform Grids}
\label{sec:MGQunif}

In this section we consider the $\mathcal V$-cycle multigrid method described in
Algorithm~\ref{alg:Vcycle} applied to solve the weighted discrete problem
\eqref{weighted_discrete_2} on quasi-uniform meshes. We consider
standard pointwise Gauss-Seidel smoothers and prove
the convergence of Algorithm \ref{alg:Vcycle} with a nearly optimal rate
up to a factor $J \eqsim |\log h_J|$.
Our main contribution is the extension of
the standard multigrid analysis \cite{Brandt1977,Brandt1984,Hackbusch1985,Xu1992siamreview}
to include weights belonging
to the Muckenhoput class $A_2(\R^N)$.
An optimal result for weights in the $A_1(\R^N)$-class is derived  in \cite{Griebel2007}.
Nevertheless, since our main motivation is the fractional Laplacian,
and the weight $y^{\alpha}\in A_2(\R^N) \setminus A_1(\R^N)$,
we need to consider the larger class $A_2(\R^N)$.

\subsection{Stability of the nodal decomposition in the weighted $L^2$-norm}
The following result states that the nodal decomposition is stable in the weighted $L^2$-norm
or, equivalently, the mass matrix for this inner product is spectrally equivalent to its diagonal.

\begin{lemma}[stability of the nodal decomposition]
\label{lem:microStab}
Let $\T \in \Tr$ be a quasi-uniform mesh, and let $v \in \V(\T)$. Then,
we have the following norm equivalence
\begin{equation}
\label{eq:weighted_stability}
  \sum _{i=1}^{\N(\T)}\|v_i\|_{L^2(\omega,D)}^2
  \lesssim
  \left\| v\right\|_{L^2(\omega,D)}^2
  \lesssim
  \sum _{i=1}^{\N(\T)}\|v_i\|_{L^2(\omega,D)}^2,
\end{equation}
where $v = \sum_{i=1}^{\N(\T)} v_{i}$ denotes the nodal decomposition for $v$, and the hidden constants in each
inequality above only depend on the dimension and the $A_2$-constant of the weight $\omega$.
\end{lemma}
\begin{proof}
Let $\widehat T \subset \R^N$ be a reference element and $\{ \widehat \phi_1, \dots, \widehat \phi_{\N_{\,\widehat T}} \}$
be its local shape functions, where $\N_{\,\widehat T}$ is the number of vertices of $\widehat T$. A standard argument shows
\[
  \widehat c_1 \left( \int_{\hat{T}} \widehat \omega \right) \sum_{i=1}^{\N_{\widehat T}} \widehat V_i^2 \leq
  \left\| \widehat v \right\|_{L^2(\widehat\omega, \widehat T)}^2
  \leq \widehat c_2 \left( \int_{\hat{T}} \widehat \omega \right) \sum_{i=1}^{\N_{\widehat T}} \widehat V_i^2,
\]
where $0 < \widehat c_1 \leq \widehat c_2$,
$\widehat v = \sum_{i=1}^{\N_{\widehat T}} \widehat V_i \widehat\phi_i$ and $\widehat\omega$ is a weight;
see \cite[Lemma~9.7]{MR2050138}. Now, given $T \in \T$, we denote by $F_{T}: \widehat T \rightarrow T$ the mapping
such that $\hat v = v \circ F_T$. Since the $A_2$ class is invariant
under isotropic dilations \cite[Proposition~2.1]{NOS2}, a scaling argument shows
\[
  \left( \int_T \omega \right)
  \sum_{i=1}^{\N_T} V_i^2 \lesssim
  \left\| v \right\|_{L^2(\omega, T)}^2
  \lesssim
  \left( \int_T \omega \right)
  \sum_{i=1}^{\N_T} V_i^2.
\]
It remains thus to show that $\int_T \omega \approx \int_T \omega \phi_i^2$.
The fact that $0 \leq \phi_i \leq 1$ yields immediately
\[
  \int_T \omega \phi_i^2\leq \int_T \omega .
\]
The converse inequality follows from the \emph{strong doubling property} of $\omega$
given in Proposition~\ref{pro:double}. In fact, setting
$E = \{ x \in T: \phi_i^2 \geq \tfrac12 \} \subset T$, we have
\[
 \int_{T} \omega \phi_i^2 \geq \int_{E} \omega \phi_i^2 \geq
 \frac{1}{2} \int_E \omega \geq \frac{1}{2C_{2,\omega}} \left( \frac{|E|}{|T|}\right)^2 \int_T \omega.
\]
Finally, notice that the supports of the nodal basis functions $\{\phi_i\}_{i=1}^{\N(\T)}$ have a finite
overlap which is independent of the refinement level, \ie for every $i = 1,\ldots, \N(\T)$, the number
$
  n(i) = \# \left\{ j : \supp\phi_i \cap \supp\phi_j \neq \emptyset \right\}
$
is uniformly bounded. We arrive at \eqref{eq:weighted_stability} summing over all the elements $T \in \T$.
\end{proof}

With the aid of the stability of the nodal decomposition, we now show a weighted inverse inequality.

\begin{lemma}[weighted inverse inequality]
\label{le:inv_est}
Let $\T \in \Tr$ be a quasi-uniform mesh, and let
$T \in \T$ and $v \in \V(\T)$. Then, we have the following inverse inequality
\begin{equation}
\label{inv_est}
  \| \nabla v \|_{L^2(\omega,T)} \lesssim h_{\T}^{-1} \| v \|_{L^2(\omega,T)}.
\end{equation}
\end{lemma}
\begin{proof}
Since $\T$ is quasi-uniform with meshsize $h_{\T}$,
we have $|\nabla \phi_i| \lesssim h_{\T}^{-1}$, and
\begin{equation*}
 \int_{T} \omega |\nabla v|^2 \lesssim h_{\T}^{-2} \sum_{i=1}^{\N_T}V_i^2\int_{T} \omega,
\end{equation*}
where, we have used the nodal decomposition of $v = \sum_{i=1}^{\N_T} V_i \phi_i$.
As in the proof of Lemma~\ref{lem:microStab}, the strong doubling property of $\omega$ yields
\[
  \int_T \omega \lesssim C_{2,\omega} \int_T \omega \phi_i^2
\]
so that we obtain
\[
  \int_{T} \omega |\nabla v|^2 \lesssim
  C_{2,\omega} h_{\T}^{-2} \sum_{i=1}^{\N} V_i^2 \int_{T} \omega \phi_i^2
  \lesssim C_{2,\omega} h_{\T}^{-2} \int_{T} \omega v^2,
\]
where, in the last step, we have used \eqref{eq:weighted_stability}. This concludes the proof.
\end{proof}


\subsection{Convergence analysis}
We now present a convergence analysis of Algorithm~\ref{alg:Vcycle} applied to solve the weighted discrete problem
\eqref{weighted_discrete_2} over quasi-uniform meshes and with standard pointwise Gauss-Seidel smoothers
\ie $\M_k = \N_k$ and $\mathcal{I}_{k,j} = \{ j \}$ for $j=1,\dots \N_k$.
The main ingredients in such analysis are the stability of the nodal decomposition
obtained in Lemma~\ref{lem:microStab}, the weighted inverse inequality of Lemma~\ref{le:inv_est},
and the properties of the quasi-interpolant introduced
in Section~\ref{sec:FEM_degenerate}. We follow \cite{MGaniso,XCN:09}.

\begin{theorem}[convergence of symmetric $\mathcal{V}$-cycle multigrid]
\label{TH:convergence_1}
Algorithm 1 with po\-int-wise Gauss-Seidel smoother is convergent with a contraction rate
\[
  \delta \leq 1 - \frac{1}{1+CJ},
\]
where $C$ is independent of the meshsize, and it
depends on the weight $\omega$ only through the constant
$C_{2,\omega}$ defined in \eqref{A_pclass}.
\end{theorem}
\begin{proof}
By the XZ identity stated in Theorem~\ref{XZ}, we only need to estimate
\begin{equation}
\label{XZidentity}
  c_0 = \sup_{\| v \|_{H_0^1(\omega, D)}=1 }
    \inf_{\sum_{k=0}^J \sum_{i=1}^{\N_k} \! v_{k,i} = v}
    \sum_{k=0}^J \sum_{i=1}^{\N_k}
    \left\|\nabla\left( P_{k,i} \sum_{(l,j) \succ (k,i)} v_{l,j} \right) \right\|_{L^2(\omega,D)}^2  \! \!,
\end{equation}
where $\succ$ stands for the so called \emph{lexicographic ordering}, \ie
\begin{equation*}
  (l,j) \succ (k,i) \Leftrightarrow
    \begin{dcases}
      l >k, \\
      l = k &  \& \ \ \ j>i.
    \end{dcases}
\end{equation*}

We recall that $k=0,\cdots J$, $j=1,\dots,\N_k$ and the operator
$P_{k,i}: \V_k \rightarrow \V_{k,i}$ is the projection with respect to the bilinear form $a$.
For $k=0,\cdots J$ we denote by $\Pi_{\T_k}$ the quasi-interpolation operator defined in
\S\ref{sub:interpolation_operator} over the mesh $\T_k$. Next,
we introduce the telescopic multilevel decomposition
\begin{equation}
\label{mdecomposition1}
v = \sum_{k=0}^J v_k, \qquad v_k = (\Pi_{\T_k} - \Pi_{\T_{k-1}})v, \qquad \Pi_{\T_{-1}} v := 0,
\end{equation}
along with the nodal decomposition
\[
  v_k = \sum_{i=1}^{\N_{\,k}} v_{k,i},
\]
for each level $k$. Consequently, the right hand side of \eqref{XZidentity} can be rewritten by
using the telescopic multilevel decomposition \eqref{mdecomposition1} as follows:
\begin{align*}
 V_{k,i}:= \sum_{(l,j) \succ (k,i)} v_{l,j} &= \sum_{l=k+1}^{J} \sum_{j=1}^{\N_k} v_{l,j}
    + \sum_{j=i+1}^{\N_k} v_{k,j} \\
  & = \sum_{l=k+1}^{J} v_{l} + \sum_{j= i +1}^{\N_k} v_{k,j}
  = v - \Pi_{\T_k} v + \sum_{j=i+1}^{\N_k} v_{k,j}.
\end{align*}
Therefore, we have
\begin{align*}
  \left\|\nabla P_{k,i}  V_{k,i}  \right\|_{L^2(\omega,D)}^2 &\lesssim
  \left\|\nabla P_{k,i} ( v - \Pi_{\T_k} v) \right\|_{L^2(\omega,D)}^2
  + \left\| \nabla  P_{k,i} \sum_{j=i+1}^{\N_k} v_{k,j} \right\|_{L^2(\omega,D)}^2 \\
  & \lesssim
  \left\|\nabla ( v - \Pi_{\T_k} v) \right\|_{L^2(\omega,D_{k,i})}^2
  +\sum_{
      \begin{subarray}{c}
        j=i+1 \\
        D_{k,i} \cap D_{k,j} \neq \emptyset
      \end{subarray}
    }^{\N_k}
    \| \nabla v_{k,j} \|_{L^2(\omega,D)}^2,
\end{align*}
where $D_{k,i} = \supp \phi_{k,i}$.
Adding over $i=1,\ldots,\N_k$,
and using the finite overlapping property
of the sets $D_{k,i}$, yields
\[
   \sum_{i=1}^{\N_k}  \sum_{
      \begin{subarray}{c}
        j=i+1 \\
        D_{k,i} \cap D_{k,j} \neq \emptyset
      \end{subarray}
    }^{\N_k}
    \| \nabla v_{k,j} \|_{L^2(\omega,D)}^2 \lesssim  \sum_{i=1}^{\N_k} \| \nabla v_{k,i} \|_{L^2(\omega,D)}^2,
\]
whence, the weighted inverse inequality \eqref{inv_est} gives
\[
  \sum_{i=1}^{\N_k} \left\|\nabla P_{k,i} V_{k,i}  \right\|_{L^2(\omega,D)}^2
   \lesssim \|\nabla ( v - \Pi_{\T_k} v) \|_{L^2(\omega,D)}^2
  + \sum_{i=1}^{\N_k} h_k^{-2} \| v_{k,i} \|_{L^2(\omega,D)}^2.
\]
We resort to the stability of the operator $\Pi_{\T_k}$, Proposition~\ref{TH:v - PivH1},
and the stability of the micro decomposition, Lemma~\ref{lem:microStab}, to arrive at
\[
  \sum_{i=1}^{\N_k} \left\|\nabla P_{k,i} V_{k,i}  \right\|_{L^2(\omega,D)}^2
  \lesssim \|\nabla v\|_{L^2(\omega,D)}^2 + h_k^{-2} \| v_k\|_{L^2(\omega,D)}^2.
\]
Since $v_k = (\Pi_{\T_k} - \Pi_{\T_{k-1}})v$, we utilize the approximation properties of $\Pi_{\T_k}$,
given in Proposition~\ref{TH:v - PivL2}, to deduce
\[
  \|  v_k \|_{L^2(\omega,D)} \leq \| v - \Pi_{\T_k} v \|_{L^2(\omega,D)}
    + \| v - \Pi_{\T_{k-1}} v \|_{L^2(\omega,D)} \lesssim h_k \| \nabla v \|_{L^2(\omega,D)}.
\]
This implies
$
\sum_{i=1}^{\N_k} \left\|\nabla P_{k,i} V_{k,i}  \right\|_{L^2(\omega,D)}^2 \lesssim \| \nabla v \|^2_{L^2(\omega,D)},
$
and adding over $k$ from $0$ to $J$ yields $c_0 \lesssim J$,
which completes the proof.
\end{proof}

\section{A multigrid method for the fractional Laplace operator on anisotropic meshes}
\label{sec:MGLaps}
As we explained in \S~\ref{sub:anisoFE},
the regularity estimate \eqref{reginy} implies the necessity
of graded meshes in the extended variable $y$. This allows us
to recover an almost-optimal error estimate for the finite element approximation
of problem \eqref{alpha_harm_intro} \cite[Theorem 5.4]{NOS}.
In fact, finite elements on quasi-uniform meshes have \emph{poor}
approximation properties for small values of the parameter $s$. The isotropic error estimates
of \cite[Theorem 5.1]{NOS} are not optimal, which makes anisotropic estimates essential. For this reason,
in this section we develop a multilevel theory for problem \eqref{alpha_harm_intro} having in mind anisotropic
partitions in the extended variable $y$ and the multilevel setting described in
Section~\ref{sec:MG_degenerate} for the
nonuniformly elliptic equation \eqref{weighted_second}.
We shall obtain nearly uniform convergence of a $\mathcal{V}$-cycle multilevel
method for the problem \eqref{alpha_harm_intro}
without any regularity assumptions. We consider line Gauss-Seidel smoothers.
The analysis is an adaptation of the results presented in ~\cite{MGaniso} for anisotropic elliptic equations, and it is
again based on the XZ identity \cite{XuZ02}.

\subsection{A multigrid algorithm with line smoothers}
As W.~Hackbusch rightfully explains \cite{Hackbusch:89}: ``\emph{the multigrid method cannot be understood as a fixed
algorithm. Usually, the components of the multigrid iteration should be adapted to the given problem, [...]
being the smoothing iteration the most delicate part of the multigrid process}''.

The success of multigrid methods for uniformly elliptic operators
is due to the fact that the smoothers are effective in reducing the nonsmooth (high frequency) components
of the error and the coarse grid corrections are effective in reducing the smooth (low frequency) components.
However, the effectiveness of both strategies depends crucially on several factors such as the anisotropy of the mesh.
A key ingredient in the design and analysis of a multigrid method
on anisotropic meshes is the use of the so called line smoothers; see \cite{AS:02,BZ:01,Hackbusch:89,S:93}.

Intuitively, when solving the $\alpha$-harmonic extension \eqref{alpha_harm_intro} on graded meshes,
the approximation from the coarse grid is dominated by the larger meshsize
in the $x$-direction and thus the coarse grid correction cannot capture the
smaller scale in the $y$-direction. One possible solution is
the use of semi-coarsening, \ie coarsening only the
$y$-direction until the meshsizes in both directions are comparable.
Another solution is the use of line smoothing, \ie solving
sub-problems restricted to one vertical line.
We shall use the latter approach which is relatively easy to implement for tensor-product meshes.

Let us describe the decomposition of $\V_J = \V(\T_{\Y_J})$ that we shall use. To do so, we follow the notation of
\S\ref{sub:multilevel}. We set $\M_k$ to be the number of interior nodes of $\T_{\Omega,k}$
and define, for $j=1,\ldots,\M_k$, the set $\calI_{k,j}$ as the collection of indices for the vertices
that lie on the line $\{\vero_j'\} \times (0,\Y)$ at the level $k$. The decomposition
is then given by \eqref{V=decomp}. This decomposition is also stable, which allows us to obtain
the appropriate anisotropic inverse inequalities; see Lemma~\ref{lem:anisoInverse} below.

Owing to the nature of the decomposition, the smoother requires the evaluation of
$A_{k,j}^{-1}$ which corresponds to the action of the operator over a vertical line. This can be efficiently realized
since the corresponding matrix is tri-diagonal.

\begin{lemma}[nodal stability and anisotropic inverse inequalities]
\label{lem:anisoInverse}
Let $\T_\Y$ be a gra\-ded tensor product grid, which is quasi-uniform in $\Omega$
and graded in the extended variable so that \eqref{graded_mesh} holds.
If  $v \in \V(\T_\Y)$ can be decomposed as $v = \sum_{j=1}^{\M_J} v_j$, then
\begin{equation}
\label{eq:stabaniso}
  \sum_{j=1}^{\M_J} \| v_j \|_{L^2(y^\alpha, \C_\Y)}^2 \lesssim
  \left\| v \right\|_{L^2(y^\alpha, \C_\Y)}^2 \lesssim
  \sum_{j=1}^{\M_J} \| v_j \|_{L^2(y^\alpha, \C_\Y)}^2.
\end{equation}
Moreover, we have the following inverse inequalities
\begin{equation}
\label{eq:invaniso}
  \| \nabla_{x'} v \|_{L^2(y^\alpha,T)} \lesssim h_K^{-1} \| v \|_{L^2(y^\alpha, T)}, \qquad
  \| \partial_y v \|_{L^2(y^\alpha,T)} \lesssim h_I^{-1} \| v \|_{L^2(y^\alpha, T)},
\end{equation}
where $T = K \times I$ is a generic element of $\T_{\Y}$.
\end{lemma}
\begin{proof}
The nodal stability \eqref{eq:stabaniso} follows along the same lines of Lemma~\ref{lem:microStab}
upon realizing that the functions $v_j = v_j(x',y)$ are defined on the vertical lines
$(\vero_j',y)$ with $y \in (0,\Y)$ and the index $j$ corresponds to a nodal decomposition in $\Omega$. Moreover,
noticing that $| \nabla_{x'} \phi_i | \lesssim h_K^{-1}$ and
$|\partial_y \phi_i | \lesssim h_I^{-1} $, we derive \eqref{eq:invaniso} inspired in Lemma~\ref{le:inv_est}.
\end{proof}

We examine Algorithm~\ref{alg:Vcycle} applied to the decomposition
\eqref{V=decomp} with exact sub-solvers on $\V_{k,j}$, \ie with line smoothers;
see~\cite[\S{III.12}]{MR1804746} and~\cite{MGaniso}.
A key observation in favor of subspaces $\{ \V_{k,j} \}_{j=1}^{\M_k}$ follows.
\begin{lemma}[nodal stability of $y$-derivatives]\label{lem:anisoy}
Under the same assumptions of Le\-mma~\ref{lem:anisoInverse} we have
\begin{equation}
\label{eq:anisoy}
  \sum_{j=1}^{\M_J} \| \partial_y v_j \|_{L^2(y^\alpha, \C_\Y)}^2 \lesssim
  \left\| \partial_y v \right\|_{L^2(y^\alpha, \C_\Y)}^2 \lesssim
  \sum_{j=1}^{\M_J} \| \partial_y v_j \|_{L^2(y^\alpha, \C_\Y)}^2.
\end{equation}
\end{lemma}
\begin{proof}
We just proceed as in Lemma \eqref{lem:microStab} with $v$ replaced by $\partial_y v = \sum_{j=1}^{\M_J} \partial_y v_j$.
\end{proof}

Exploiting Theorem~\ref{XZ}, the properties of the quasi-interpolation operator $\Pi_{\T_k}$ defined in
\S\ref{sub:interpolation_operator}, and Lemmas~\ref{lem:anisoInverse} and \ref{lem:anisoy},
we obtain the nearly uniform convergence of the symmetric
$\mathcal V$-cycle multigrid method. We follow \cite{MGaniso,XCN:09}.

\begin{theorem}[convergence of multigrid methods with line smoothers]
\label{TH:convergence_2}
The symmetric $\mathcal{V}$-cycle multigrid method with line smoothing converges with a contraction rate
\[
  \delta \leq 1 - \frac{1}{1+CJ},
\]
where $C$ is independent of the number of degrees of freedom. The constant $C$
depends on the weight $y^{\alpha}$ only through the constant $C_{2,y^{\alpha}}$, and on $s$ like
$C \approx \gamma $, where $\gamma$ is the parameter that defines the graded mesh \eqref{graded_mesh}.
\end{theorem}
\begin{proof}
We use the XZ identity \eqref{XZ} and modify the arguments in the proof of Theorem~\ref{TH:convergence_1}.
We introduce the telescopic multilevel decomposition
\begin{equation}\label{mdecomposition}
  v = \sum_{k=0}^J v_k, \qquad v_k = (\Pi_{\T_{\Y,k}} - \Pi_{\T_{\Y,k-1}})v, \qquad \Pi_{\T_{\Y,-1}} v := 0,
\end{equation}
along with the line decomposition
\[
  v_k = \sum_{j=1}^{\M_k} v_{k,j}.
\]
Following the same arguments developed in the proof of
Theorem~\ref{TH:convergence_1},
and denoting $V_{k,i} = \sum_{(l,j) \succ (k,i)} v_{l,j}$, we arrive at the inequality
\begin{equation}
\label{aux_conv_1}
    \sum_{i=1}^{{\M_k}}
    \left\|\nabla P_{k,i} V_{k,i} \right\|_{L^2(y^{\alpha},\C_{\Y})}^2
    \lesssim
    \|\nabla ( v - \Pi_{\T_{\Y,k}} v) \|_{L^2(y^{\alpha},\C_{\Y})}^2
    +
    \sum_{j=1}^{\M_k} \| {\nabla v_{k,j}} \|_{L^2(y^{\alpha},\C_{\Y})}^2,
\end{equation}
where we have used the finite overlapping property of the sets $\mathcal{I}_{k,j}$; see \S~\ref{sub:multilevel}.
It remains to estimate both terms in \eqref{aux_conv_1}. The stability of the quasi-interpolant
$\Pi_{\T_{\Y,k}}$ stated in \eqref{dPIcontinuous} (see also \cite[Theorems 4.7 and 4.8]{NOS}
and \cite[Lemma 5.1]{NOS2}) yields
\begin{equation}
\label{aux_conv_*}
 \|\nabla ( v - \Pi_{\T_{\Y,k}} v) \|_{L^2(y^{\alpha},\C_{\Y})} \lesssim \|\nabla v \|_{L^2(y^{\alpha},\C_{\Y})}.
\end{equation}
To estimate the second term in \eqref{aux_conv_1} we begin by noticing that
\begin{equation}
\label{aux_conv_2}
  \sum_{j=1}^{\M_k} \| \nabla v_{k,j} \|_{L^2(y^{\alpha},\C_{\Y})}^2 =
  \sum_{j=1}^{\M_k} \| \nabla_{x'} v_{k,j} \|_{L^2(y^{\alpha},\C_{\Y})}^2
  +
  \sum_{j=1}^{\M_k} \| \partial_{y} v_{k,j} \|_{L^2(y^{\alpha},\C_{\Y})}^2.
\end{equation}
The first term is estimated via the first weighted inverse inequality \eqref{eq:invaniso}
and the stability of the nodal decomposition \eqref{eq:stabaniso}, that is
\begin{equation}
\label{aux_conv_3}
  \sum_{j=1}^{\M_k} \| \nabla_{x'} v_{k,j} \|_{L^2(y^{\alpha},\C_{\Y})}^2
  \lesssim \sum_{j=1}^{\M_k} {h'_k}^{-2}\|  v_{k,j} \|_{L^2(y^{\alpha},\C_{\Y})}^2
  \lesssim {h'_k}^{-2}\|  v_k \|_{L^2(y^{\alpha},\C_{\Y})}^2,
\end{equation}
where ${h'_k}$ denotes the meshsize in the $x'$ direction at level $k$.
The approximation property of $\Pi_{\T_{\Y,k}}$ stated in Lemma \ref{le:l2aniso}
(see also \cite[Theorem 5.7]{NOS2}) and the definition of $v_k$ yield
\begin{align*}
  \|  v_k \|_{L^2(y^{\alpha},\C_{\Y})}
 & \leq \| v - \Pi_{\T_{\Y,k}} v \|_{L^2(y^{\alpha},\C_{\Y})}
    + \| v - \Pi_{\T_{\Y,k-1}} v \|_{L^2(y^{\alpha},\C_{\Y})}\\
 & \lesssim
  {h'_k}\| \nabla_{x'}v \|_{L^2(y^{\alpha},\C_{\Y})}
  + {h''_k}\|  \partial_{y} v \|_{L^2(y^{\alpha},\C_{\Y})}
\end{align*}
where ${h''_k}$ denotes the \emph{maximal} meshsize in the $y$ direction at level $k$.
Using \eqref{eq:mesh_rels} we see that ${h''_k} \lesssim \gamma{h'_k}$, and replacing the estimate above
in \eqref{aux_conv_3}, we obtain
\begin{equation}
\label{aux_conv_4}
  \sum_{j=1}^{\M_k} \| \nabla_{x'} v_{k,j} \|_{L^2(y^{\alpha},\C_{\Y})}^2
  \lesssim \| \nabla v\|_{L^2(y^{\alpha},\C_{\Y})}^2,
\end{equation}
which bounds the first term in \eqref{aux_conv_2}.
To estimate the second term, we resort to Lemma~\ref{lem:anisoy}, namely
%
\begin{equation}
\label{aux_conv_5}
 \sum_{j=1}^{\M_k} \| \partial_{y} v_{k,j} \|_{L^2(y^{\alpha},\C_{\Y})}^2 \lesssim
  \| \partial_{y} v_{k} \|_{L^2(y^{\alpha},\C_{\Y})}^2.
\end{equation}
Finally, inequalities \eqref{aux_conv_4} and \eqref{aux_conv_5} allow us to conclude
\begin{equation*}
  \sum_{j=1}^{\M_k} \| \nabla v_{k,j} \|_{L^2(y^{\alpha},\C_{\Y})}^2
  \lesssim \| \nabla v \|_{L^2(y^{\alpha},\C_{\Y})}^2,
\end{equation*}
which together with  \eqref{aux_conv_*} yields the desired result after summing over $k$.
\end{proof}

\begin{remark}[dependence on $s$]
We point out the use of \eqref{eq:mesh_rels}, which in turn implies
${h''_k} \lesssim \gamma{h'_k}$, to derive \eqref{aux_conv_4}. This translates into $C \approx \gamma$
in Theorem~\ref{TH:convergence_2} and, since $\gamma > 3/(1-\alpha)=3/(2s)$, in deterioration of the contraction
factor as $s \downarrow 0$. We explore a remedy in \S~\ref{sub:meshmodif}.
\end{remark}

\section{Numerical Illustrations}
\label{sec:Numerics}

In this section, we present numerical experiments to support our theoretical findings.
We consider two examples:
\begin{enumerate}[(\thesection.1)]
 \item \quad $n=1$, \quad $\Omega = (0,1)$, \quad $u = \sin(3\pi x)$,
 \item \quad $n=2$, \quad $\Omega = (0,1)^2$, \quad $u = \sin (2\pi x_1)\sin(2\pi x_2)$,
\end{enumerate}
and $\Y=1$. The length $\Y$ of the cylinder in the
extended direction is fixed, as discussed in \cite{NOS}, so that it captures the exponential decay of the solution.
All of our algorithms are implemented based on the MATLAB$^\copyright$ software package {\it{i}}FEM~\cite{Chen.L2008c}.

\subsection{Multigrid with line smoothers on graded meshes}
\label{sub:NumExpMGaniso}

We partition $\Omega$ into a uniform grid of size $h_{\T_{\Omega}}$,
and we construct a graded mesh in the extended direction
using the mapping \eqref{graded_mesh} with parameter $\gamma = \frac3{2s} +0.1$ and $M = \frac1h$. Some sample meshes
are shown in Figure~\ref{fig:1ds08}. The mesh points are ordered column-wise so that the indices associated to vertical
lines are easily accessible. Starting from $h_{\T_0}=\frac14$ we obtain a sequence of meshes by halving
the meshsize of $\Omega$ and
applying the mapping \eqref{graded_mesh} in the extended direction with double number of mesh points.

\begin{figure}[ht]
\begin{center}
  \includegraphics[scale=0.4]{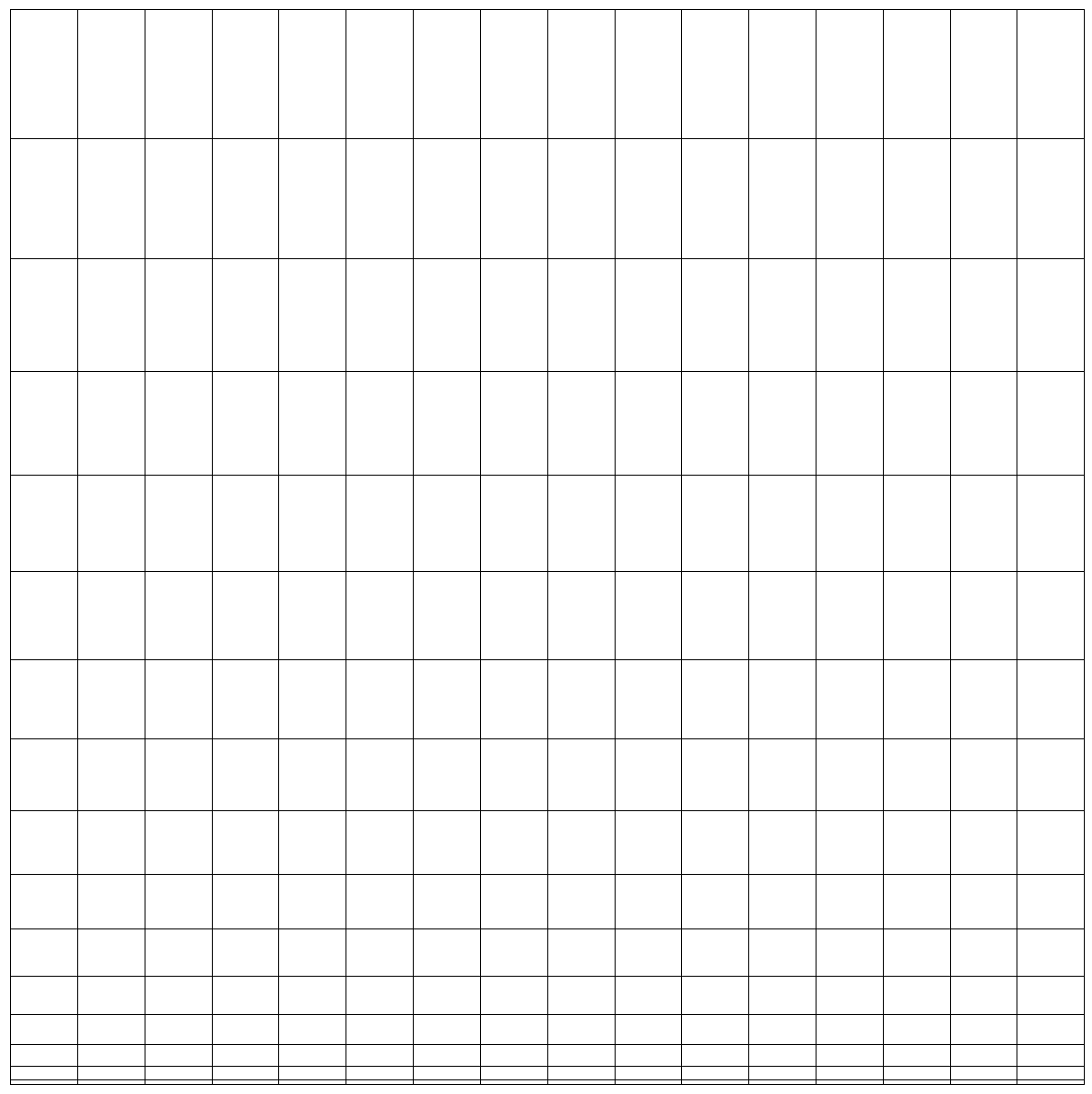}
  \hfil
  \includegraphics[scale=0.4]{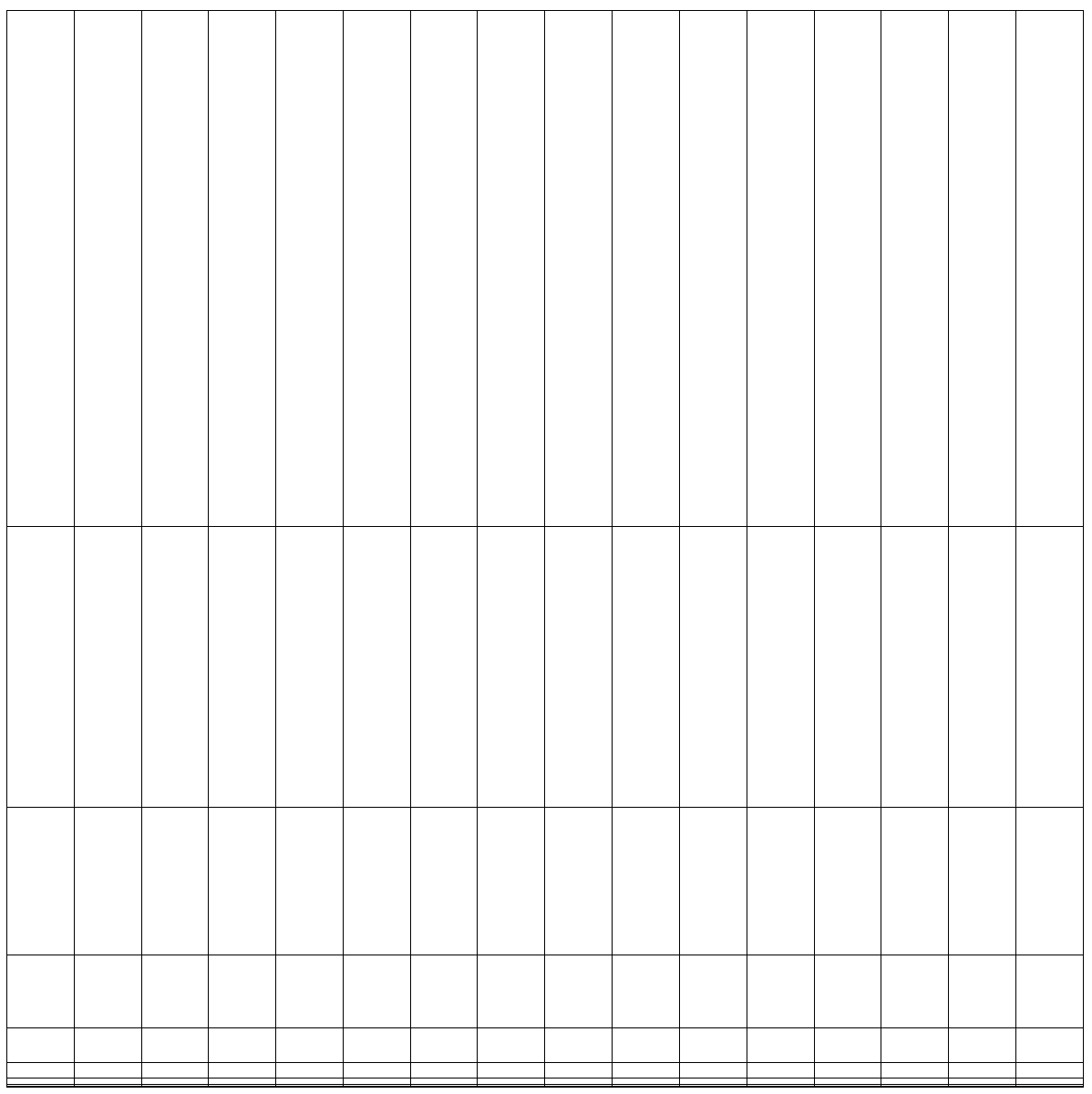}
\end{center}
\caption{
Graded meshes for the cylinder $\mathcal C_{\Y} = (0,1)\times (0,1)$.
In both cases, the mesh in $\Omega$ is uniform and of size $h_{\T_0}=\frac1{16}$.
The meshes in the extended direction are graded according to \eqref{graded_mesh} with $M=\frac1{h}$ and
$\gamma= \frac3{2s} +0.1$.  The \emph{left} mesh is for $s=0.8$ and the \emph{right} for $s=0.15$.}
\label{fig:1ds08}
\end{figure}

We assemble the matrix corresponding to the finite element discretization of \eqref{weighted_discrete_2} on each level.
The natural embedding $\mathbb V(\T_k) \to \mathbb V(\T_{k+1})$ for $k=0,\ldots,J-1$ gives us the prolongation matrix
between two consecutive levels. Notice that the prolongation in the $x'$-direction is obtained by standard averaging,
while in the extended direction the weights must be modified to take into account the grading of the mesh. The restriction
matrix is taken as the transpose of the prolongation matrix.

As discussed in Section~\ref{sec:MGLaps} we must use vertical line smoothers to attain efficiency of the
multigrid method. The tri-diagonal sub-matrix corresponding to one vertical line is inverted exactly by using the
built-in direct solver in MATLAB$^\copyright$.
Red-black ordering of the indices in the $x'$-direction is used to further improve the
efficiency of the line smoothers.
We perform three pre- and post-smoothing steps, \ie $m=3$.
We start with the zero initial guess and use as exit criterion that the $\ell^2$-norm of the
relative residual is smaller than $10^{-7}$.

Tables~\ref{table:1d} and \ref{table:2d} show the number of iterations for the implemented multigrid method for the one
and two dimensional problems, respectively. As we see, the method converges almost uniformly with respect
to the number of degrees of freedom. Notice that the number of iterations for $s=0.15$ is significantly larger than that
for the remaining tested cases. This can be explained by the fact that, as Theorem~\ref{TH:convergence_2} states,
the contraction factor depends on $\gamma \approx \tfrac1s$ and thus, we observe a preasymptotic regime where
the number of iterations grows. This is exactly the case for the one dimensional problem and we would expect
a similar behavior in the two dimensional case. However, since the extended problem is now in three dimensions, the
size of the problems grows rather quickly and thus our computational resources were not sufficient to deal with
the cases $h_{\T_{\Omega}} = \tfrac1{256}$ and $h_{\T_{\Omega}}=\tfrac1{512}$. In \S\ref{sub:meshmodif} we will propose a modification of the graded mesh in the extended direction to
address this issue.

\begin{table}[htdp]
  \begin{center}
    \begin{tabular}{|l|r|c|c|c|c|}
      \hline
      $h_{\T_{\Omega}}$           & DOFs     & $s=0.15$  & $s=0.3$ & $s=0.6$ & $s=0.8$ \\
      \hline
      $\tfrac1{16}$  & 289     & 7         & 6       & 5       & 5 \\
      \hline
      $\tfrac1{32}$  & 1,089   & 13        & 9       & 6       & 6 \\
      \hline
      $\tfrac1{64}$  & 4,225   & 25        & 10      & 6       & 6 \\
      \hline
      $\tfrac1{128}$ & 16,641  & 33        & 11      & 6       & 6 \\
      \hline
      $\tfrac1{256}$ & 66,049  & 37        & 10      & 6       & 6 \\
      \hline
      $\tfrac1{512}$ & 263,169 & 38        & 10      & 6       & 7 \\
      \hline
    \end{tabular}
  \end{center}
\vspace{0.2cm}
\caption{
Number of iterations for a multigrid method for the one dimensional fractional Laplacian using a line smoother
in the extended direction. The mesh in $\Omega$ is uniform of size $h_{\T_{\Omega}}$. The mesh in the extended
direction is graded according to \eqref{graded_mesh}.}
\label{table:1d}
\end{table}

\begin{table}[htdp]
  \begin{center}\begin{tabular}{|l|r|c|c|c|c|}
  \hline
  $h_{\T_{\Omega}}$            & DOFs       & $s=0.15$  & $s=0.3$ & $s=0.6$ & $s=0.8$ \\
  \hline
  $\tfrac1{16}$  & 4,913      & 10        & 7       & 6       & 5 \\
  \hline
  $\tfrac1{32}$  & 35,937     & 19        & 8       & 6       & 6 \\
  \hline
  $\tfrac1{64}$  & 274,625    & 34        & 9       & 6       & 6 \\
  \hline
  $\tfrac1{128}$ & 2,146,689  & 47        & 9       & 6       & 6\\
  \hline
  \end{tabular}
  \end{center}
\vspace{0.2cm}
\caption{
Number of iterations for a multigrid method for the two dimensional fractional Laplacian using a line smoother
in the extended direction. The mesh in $\Omega$ is uniform of size $h_{\T_{\Omega}}$. The mesh in the extended direction
is graded according to \eqref{graded_mesh}.}
\label{table:2d}
\end{table}

We also tested a point Gauss-Seidel smoother for the one dimensional case $\Omega = (0,1)$.
Except for the trivial case $h_{\T_{\Omega}} = 1/16$, the corresponding $\mathcal V$-cycle is not able to achieve
the desired accuracy in $200$ iterations.

\subsection{Multigrid methods on quasi-uniform meshes}

Even though the approximation of the Caffarelli-Silvestre extension of the fractional Laplace operator on quasi-uniform
meshes in the extended direction is suboptimal, let us use this problem to illustrate the convergence properties of the
multilevel method, developed in Section~\ref{sec:MGQunif}, for general $A_2$ weights. The setting is the same as in the
previous subsection but we use a point-wise Gauss-Seidel smoother. Tables~\ref{table:1duniform} and \ref{table:2duniform}
show the number of iterations with respect to the number of degrees of freedom and $s$. We see that the convergence is
almost uniform with respect to the number of unknowns as well as the parameter $s \in (0,1)$.

\begin{table}[ht]
  \begin{center}
  \begin{tabular}{|l|r|c|c|c|c|}
  \hline
  $h_{\T_{\Omega}}$             & DOFs    & $s=0.15$  & $s=0.3$ & $s=0.6$ & $s=0.8$ \\
  \hline
  $\tfrac1{16}$   & 289     & 12        & 13      & 13      & 14 \\
  \hline
  $\tfrac1{32}$   & 1,089   & 15        & 15      & 15      & 17 \\
  \hline
  $\tfrac1{64}$   & 4,225   & 15        & 16      & 16      & 17 \\
  \hline
  $\tfrac1{128}$  & 16,641  & 15        & 16      & 16      & 18  \\
  \hline
  $\tfrac1{256}$  & 66,049  & 15        & 15      & 16      & 18  \\
  \hline
  $\tfrac1{512}$  & 263,169 & 15        & 15      & 16      & 18 \\
  \hline
  \end{tabular}
  \end{center}
\vspace{0.2cm}
\caption{Number of iterations for a multigrid method with point-wise Gauss-Seidel smoothers on uniform meshes for the one dimensional fractional Laplacian.}
\label{table:1duniform}
\end{table}

\begin{table}[ht]
  \begin{center}
  \begin{tabular}{|l|r|c|c|c|c|}
  \hline
  $h_{\T_{\Omega}}$             & DOFs      & $s=0.15$  & $s=0.3$ & $s=0.6$ & $s=0.8$ \\
  \hline
  $\tfrac1{16}$   & 4,913     & 13        & 12      & 13      & 15 \\
  \hline
  $\tfrac1{32}$   & 35,937    & 15        & 15      & 15      & 17 \\
  \hline
  $\tfrac1{64}$   & 274,625   & 15        & 16      & 16      & 18 \\
  \hline
  $\tfrac1{128}$  & 2,146,689 & 15        & 16      & 16      & 19 \\
  \hline
  \end{tabular}
  \end{center}
\vspace{0.2cm}
\caption{Number of iterations for a multigrid method with point-wise Gauss-Seidel smoothers on uniform meshes for the two dimensional fractional Laplacian.}
\label{table:2duniform}
\end{table}

\subsection{Modified mesh grading}
\label{sub:meshmodif}

Examining the proof of Theorem~\ref{TH:convergence_2}, we realize that the critical step \eqref{aux_conv_4}
consists in the application of inequality \eqref{eq:mesh_rels}, namely $h_k'' \lesssim \gamma h_k'$,
which deteriorates as  $s$ becomes small because
  $\gamma>3/(1-\alpha) = 3/(2s)$. Numerically, this effect can be seen in Tables \ref{table:1d} and \ref{table:2d} where, for instance, the number
of iterations needed for $s=0.15$ is significantly larger than that for all the other tested values;
see the right mesh
for $s=0.15$ in Figure \ref{fig:1ds08}.
As a result, the contraction rate of Theorem~\ref{TH:convergence_2}
becomes $1-1/(1+C{\gamma}J)$.
Here we explore computationally how to overcome this issue. We construct a mesh such that the maximum meshsize in the
extended direction is uniformly bounded, with respect to $s$, by the
uniform meshsize in the $x'$-direction without changing
the ratio of degrees of freedom in $\Omega$ and the extended direction by more than a constant.

Let us begin with some heuristic motivation.
In order to control the aspect ratio $h_k''/ h_k'$ uniformly on $s \in (0,1)$,
we may apply some extra refinements to the largest elements in the $y$ direction,
increasing the number of degrees of freedom of $\T_{\Y}$ just by a constant.
We denote by $\tilde{\T_{\Y}}$ the resulting mesh and we notice that
$\V(\T_{\Y}) \subset \V(\tilde{\T}_{\Y})$. Thus, Galerkin
orthogonality implies
\begin{align*}
 \| \nabla (v - V_{\tilde{\T_{\Y}}} ) \|_{L^2(y^\alpha,\C_{\Y})}
& = \inf\left\{ \| \nabla (v - W) \|_{L^2(y^\alpha,\C_{\Y})} : W \in \V(\tilde{\T_{\Y}}) \right \}\\
&\leq \| \nabla (v - V_{\T_{\Y}})\|_{L^2(y^\alpha,\C_{\Y})} \lesssim \left(\#\T_\Y\right)^{-\frac1{n+1}}
\approx (\# \tilde{ \T_{\Y}} )^{-\frac1{n+1}}.
\end{align*}
We build on this idea through a modification of the mapping function below.

Let $F: (0,1)\to (0, \Y)$ be an increasing and differentiable function such that $F(0)=0$ and $F(1) = \Y$. By mapping a
uniform grid of $(0,1)$ via the function $F$, we can construct a graded mesh with mesh points given by
$y_l = F(l/M)$ for $l=1,..., M$; for instance, $F(\xi) = \Y \xi^\gamma$ yields \eqref{graded_mesh}. The mean value theorem
implies
\[
  y_{l+1}- y_l = \frac{F'(c_l)}{M} \leq \frac1M \max\left\{ |F'(\xi)|: \xi \in \left[\frac{l}M,\frac{l+1}M \right] \right\},
\]
which shows that the map of \eqref{graded_mesh} is not uniformly bounded with respect to $s$.

For this reason, we instead consider the following construction: Let $(\xi_\star, y_\star)\in (0,1)^2$, which we will call the \emph{transition point}, and define the mapping
\[
  F(\xi)=
  \begin{dcases}
    y_\star \Y \left( \frac{\xi}{\xi_\star} \right )^{\gamma},  & 0<\xi\leq \xi_\star, \\
    \Y\left( \frac{1-y_\star}{1-\xi_\star}(\xi - \xi_\star) + y_\star \right), & \xi_\star<\xi <1.
  \end{dcases}
\]

Over the interval $(0,\xi_\star)$ the mapping $F$ defines the same type of graded mesh but, over $(\xi_\star,1)$
it defines a uniform mesh. Let us now choose the transition point to obtain a bound on the derivative of $F$. We have
\begin{equation}
\label{modifiedF}
  F'(\xi) =
  \begin{dcases}
    \gamma \Y \frac{y_\star}{\xi_\star} \left( \frac{\xi}{\xi_\star} \right )^{\gamma-1}, & 0 < \xi \leq \xi_\star, \\
    \Y \frac{1-y_\star}{1-\xi_\star}, & \xi_\star < \xi <1,
  \end{dcases}
\end{equation}
so that
\[
 \mathcal F := \max_{\xi \in [0,1]} |F'(\xi)| = \Y \max \left \{\gamma \frac{y_\star}{\xi_\star}, \frac{1-y_\star}{1-\xi_\star} \right \}.
\]
Given $\xi_\star$ we choose $y_\star$ to have $\gamma \frac{y_\star}{\xi_\star} = \frac{1-y_\star}{1-\xi_\star}$, \ie
\begin{equation*}
  y_\star = \frac{1}{1+ \gamma \tfrac{1-\xi_\star}{\xi_\star}}.
\end{equation*}
this immediately yields $F \in \mathcal{C}^1([0,1])$ and, more importantly,
\[
 \mathcal F = \gamma\Y \frac{y_\star}{\xi_\star}
   = \Y \frac{\gamma}{\xi_\star + (1-\xi_\star)\gamma} \leq \Y \frac{1}{1-\xi_\star}.
\]
We can now choose $\xi_\star$ to gain control of $\mathcal F$. For instance, $\xi_\star = 0.5$ gives us that
$\mathcal F \leq 2$ and $\xi_\star = 0.75$ that $\mathcal F\leq 4$. In the
experiments presented below we choose $\xi_\star = 0.75$. The theory presented in \S~\ref{sec:MGLaps}
still applies.

\begin{figure}[htbp]
  \includegraphics[scale=0.4]{1dmeshs015.pdf}
  \hfil
  \includegraphics[scale=0.4]{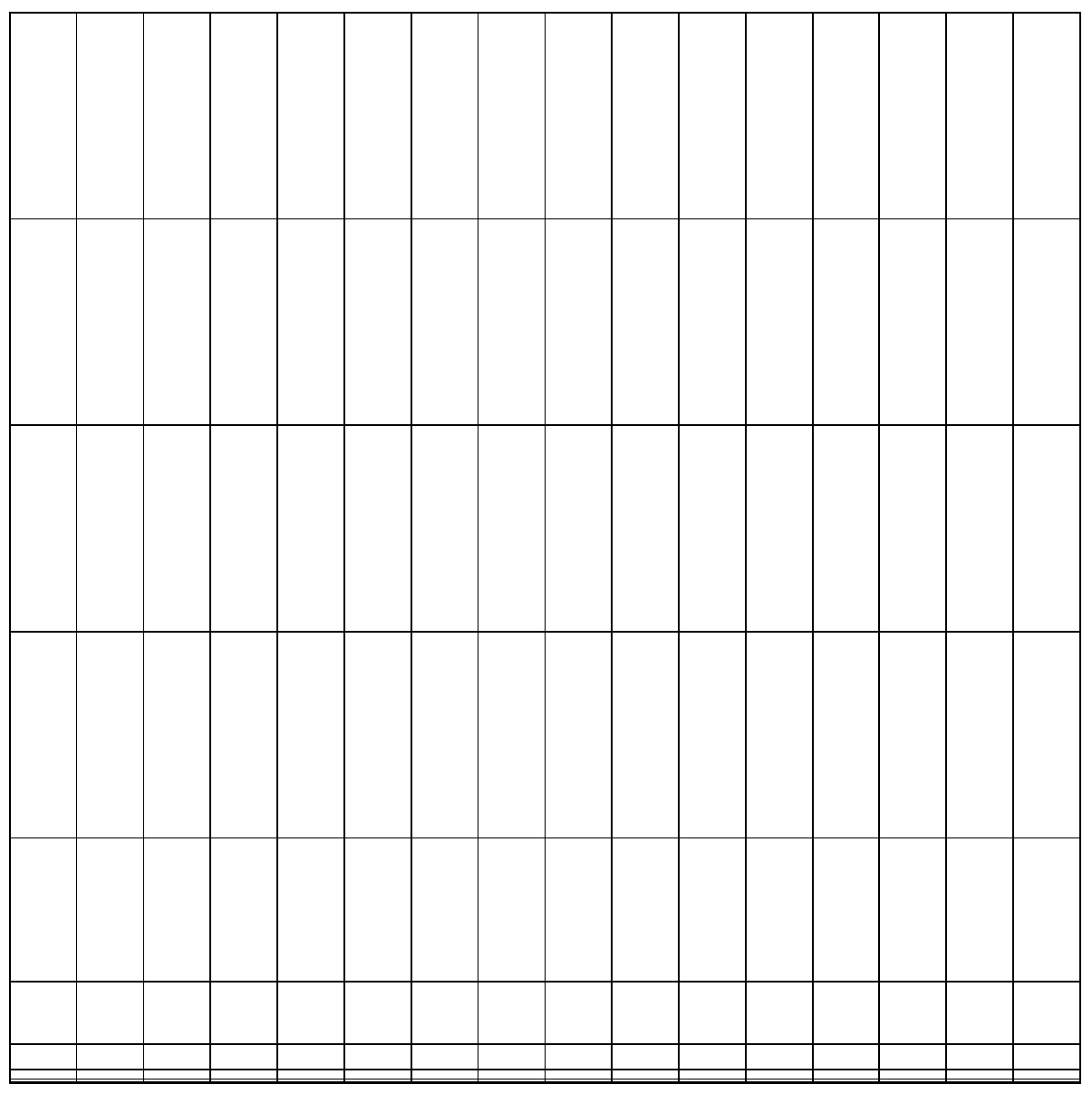}
\caption{Graded meshes for the extended domain $\mathcal C_{\Y} = (0,1)\times (0,1)$, $h_{\T_{\Omega}}=\tfrac1{16}$ and $s=0.15$.
\emph{Left}: The grading is according to \eqref{graded_mesh}. \emph{Right}: The grading is given by the map \eqref{modifiedF}.}
\label{fig:suki}
\end{figure}

\begin{figure}
  \includegraphics[scale=0.4]{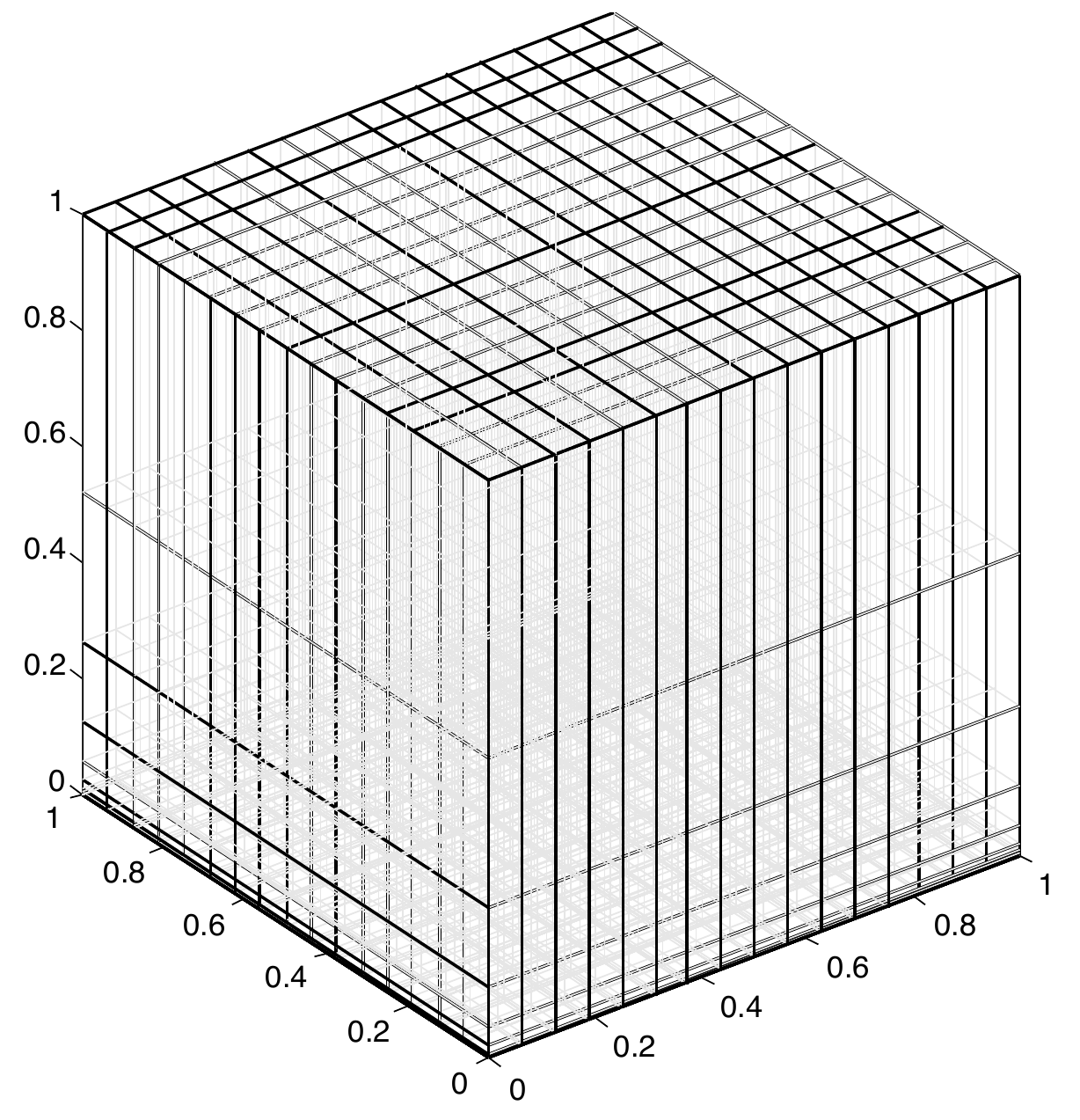}
  \hfil
  \includegraphics[scale=0.4]{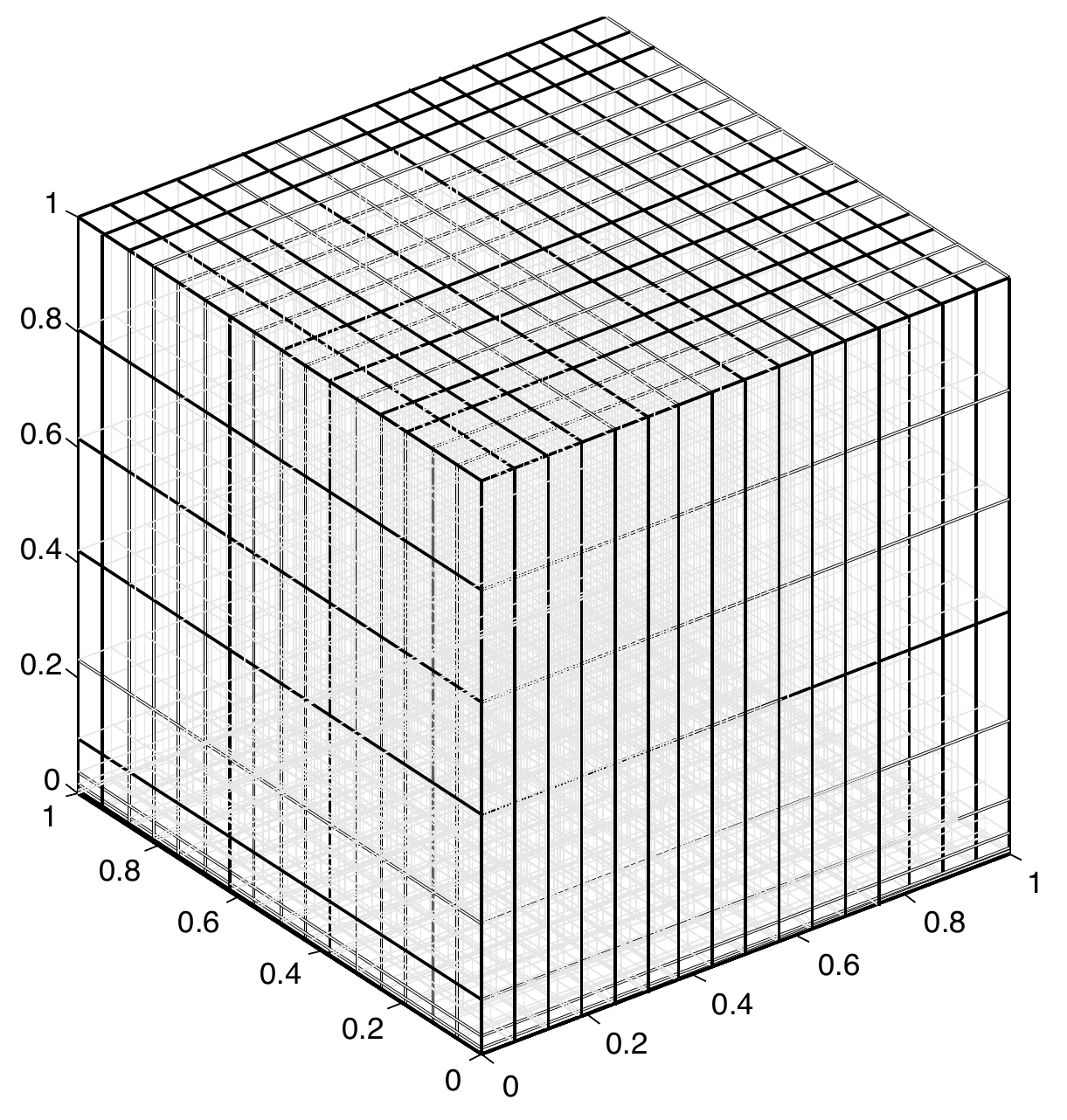}
\caption{Graded meshes for the extended domain $\mathcal C_{\Y} = (0,1)^2\times (0,1)$, $h_{\T_{\Omega}}=\tfrac1{16}$ and $s=0.15$. \emph{Left}: The grading is according to \eqref{graded_mesh}. \emph{Right}: The grading is given by the map \eqref{modifiedF}. }
\label{fig:urodi}
\end{figure}

To better visualize the effect of this modification Figures~\ref{fig:suki} and \ref{fig:urodi}, show the
original graded mesh, defined by \eqref{graded_mesh}, and the modified one obtained using \eqref{modifiedF}, in two
and three dimensions, respectively. The modified graded meshes have asymptotically the same distribution of points
near the bottom part of the cylinder and so they are also capable of capturing the singular behavior of the solution
$\ve$. However, near the top part, the aspect ratio is uniformly controlled by a factor $4$.
The modified mesh is only applied for $\gamma > 4$.
For $s=0.3$, $0.6$ and $0.8$, no modification is needed in the original mesh.

Upon constructing a mesh with this modification, we can develop a $\mathcal V$-cycle multigrid solver with vertical
line smoothers. Comparisons of this approach with the
setting of \S~\ref{sub:NumExpMGaniso} are shown
in Tables~\ref{tab:1Dcomparison} and \ref{tab:2Dcomparison}. From them we can conclude that the strong anisotropic
behavior of the mesh grading \eqref{graded_mesh} affects the performance of the $\mathcal V$-cycle multigrid with
vertical line smoothers. For the original graded meshes, there is a preasymptotic regime where the number of
iterations increases faster than $\log J$. The modification of the mesh proposed in \eqref{modifiedF} allows us to
obtain an almost uniform number of iterations for all problem sizes without sacrificing the accuracy of the method.
This is also evidenced by the computational time required to solve a problem with a fixed number of degrees of freedom.

\begin{table}[ht]
  \begin{center}
    \begin{tabular}{|l|r|c|c|c|c|c|c|}
    \hline
    $h_{\T_{\Omega}}$             & DOFs    &I(o) & I(m)  & E(o)    & E(m)    & CPU(o)  & CPU(m)  \\
    \hline
    $\tfrac1{16}$   & 289     & 7   & 7     & 0.1556  & 0.1739  & 0.0209  & 0.0554  \\
    \hline
    $\tfrac1{32}$   & 1,089   & 13  & 9     & 0.0828  & 0.0937  & 0.0664  & 0.0985  \\
    \hline
    $\tfrac1{64}$   & 4,225   & 25  & 10    & 0.0426  & 0.0485  & 0.2337  & 0.2720  \\
    \hline
    $\tfrac1{128}$  & 16,641  & 33  & 10    & 0.0216  & 0.0246  & 0.9041  & 0.4496  \\
    \hline
    $\tfrac1{256}$  & 66,049  & 37  & 11    & 0.0109  & 0.0124  & 4.8168  & 1.7051 \\
    \hline
    $\tfrac1{512}$  & 263,169 & 38  & 11    & 0.0055  & 0.0062  & 25.1351 & 7.3439  \\
    \hline
    \end{tabular}
  \end{center}
\vspace{0.2cm}
\caption{Comparison of the multilevel solver with vertical line smoother over two graded meshes for the one
dimensional fractional Laplacian, $s=0.15$. \emph{Legend}: The original mesh, given by \eqref{graded_mesh} is
denoted by o, whereas the modification proposed in \eqref{modifiedF} is denoted by m.
I -- number of iterations, E -- error in the energy norm, CPU -- cpu time ($s$).}
\label{tab:1Dcomparison}
\end{table}

\begin{table}[ht]
  \begin{center}
  \begin{tabular}{|l|r|c|c|c|c|c|c|}
  \hline
  $h_{\T_{\Omega}}$             & DOFs      & I(o)  & I(m)  & E(o)    & E(m)    & CPU(o)  & CPU(m)  \\
  \hline
  $\tfrac1{16}$   & 4,913     & 10    & 8     & 0.1070  & 0.1198  & 0.41    & 0.31 \\
  \hline
  $\tfrac1{32}$   & 35,937    & 19    & 11    & 0.0570  & 0.0646  & 4.76    & 2.95 \\
  \hline
  $\tfrac1{64}$   & 274,625   & 34    & 12    & 0.0294  & 0.0334  & 82.56   & 31.48 \\
  \hline
  $\tfrac1{128}$  & 2,146,689 & 47    & 13    & 0.0149  & 0.0170  & 892.65  & 269.63\\
  \hline
  \end{tabular}
  \end{center}
\vspace{0.2cm}
\caption{Comparison of the multilevel solver with vertical line smoother over two graded meshes for the two dimensional fractional Laplacian, $s=0.15$.
\emph{Legend}: The original mesh, given by \eqref{graded_mesh} is denoted by o, whereas the modification proposed in \eqref{modifiedF} is denoted by m.
I -- number of iterations, E -- error in the energy norm, CPU -- cpu time ($s$).}
\label{tab:2Dcomparison}
\end{table}

\bibliographystyle{plain}
\bibliography{biblio}
\end{document}